\documentclass[11 pt]{article}
\title{On the second homology group of the Torelli subgroup of $\text{Aut}(F_n)$\vspace{-12pt}}
\author{Matthew Day\thanks{Supported in part by NSF grant DMS-1206981}\ \ and Andrew Putman\thanks{Supported in part by NSF grant DMS-1255350 and the Alfred P.\ Sloan Foundation}\vspace{-8pt}}
\date{}

\newcommand{\arxiv}[1]{\href{http://arxiv.org/abs/#1}{{\tt arXiv:#1}}}

\usepackage{amsmath,amssymb,amsthm,amscd,amsfonts}
\usepackage{epsfig,pinlabel,paralist}
\usepackage[hmargin=1.25in, vmargin=1in]{geometry}
\usepackage[font=small,format=plain,labelfont=bf,up,textfont=it,up]{caption}
\usepackage{type1cm}
\usepackage{calc}
\usepackage{enumerate}
\usepackage{url}
\usepackage{bm}
\usepackage{booktabs}
\usepackage{microtype}
\usepackage{titling}
\usepackage[all,cmtip]{xy}
\usepackage{mathabx}

\usepackage[dvipdfm, bookmarks, bookmarksdepth=2, colorlinks=true, linkcolor=blue, citecolor=blue, urlcolor=blue]{hyperref}

\theoremstyle{plain}
\newtheorem{theorem}{Theorem}[section]
\newtheorem{maintheorem}{Theorem}
\newtheorem{proposition}[theorem]{Proposition}
\newtheorem{lemma}[theorem]{Lemma}

\newtheorem{claims}{Claim}
\newtheorem{step}{Step}

\newcommand\BeginClaims{\setcounter{claims}{0}}

\newcommand\BeginSteps{\setcounter{step}{0}}

\theoremstyle{definition}

\newtheorem{definition}[theorem]{Definition}

\theoremstyle{remark}
\newtheorem{remark}[theorem]{Remark}

\newtheorem*{example}{Example}

\DeclareMathOperator{\Hom}{Hom}

\DeclareMathOperator{\Ker}{ker}

\DeclareMathOperator{\IA}{IA}

\DeclareMathOperator{\GL}{GL}
\DeclareMathOperator{\SL}{SL}

\newcommand\R{\ensuremath{\mathbb{R}}}

\newcommand\Z{\ensuremath{\mathbb{Z}}}
\newcommand\Q{\ensuremath{\mathbb{Q}}}

\DeclareMathOperator{\HH}{H}

\DeclareMathOperator{\Aut}{Aut}
\DeclareMathOperator{\End}{End}

\newcommand\Set[2]{\ensuremath{\{\text{#1 $|$ #2}\}}}
\newcommand\Pres[2]{\ensuremath{\langle \text{$#1$ $|$ $#2$} \rangle}}
\newcommand\LPres[3]{\ensuremath{\langle \text{$#1$ $|$ $#2$ $|$ $#3$} \rangle}}
\newcommand\llangle{\ensuremath{\langle \hspace{-3pt} \langle}}
\newcommand\rrangle{\ensuremath{\rangle \hspace{-3pt} \rangle}}
\DeclareMathOperator{\Mag}{MA}

\newcommand\Quotient{\ensuremath{\mathcal{Q}}}
\newcommand{\BigFreeProd}{\mathop{\mbox{\Huge{$\ast$}}}}
\newcommand\Mul[2]{\ensuremath{\text{M}_{#1,#2}}}
\newcommand\Mulcomm[3]{\ensuremath{\text{M}_{#1,[#2,#3]}}}
\newcommand\Con[2]{\ensuremath{\text{C}_{#1,#2}}}
\newcommand\QMulcomm[3]{\ensuremath{\mathfrak{M}_{#1,[#2,#3]}}}
\newcommand\QCon[2]{\ensuremath{\mathfrak{C}_{#1,#2}}}
\newcommand\ttheta{\ensuremath{\theta}}
\newcommand\tphi{\ensuremath{\phi}}
\newcommand\Conj[1]{\ensuremath{\ldbrack #1 \rdbrack}}
\newcommand\Lax[1]{\ensuremath{( #1 )_{\pm}}}
\newcommand\Bases{\ensuremath{\mathcal{B}}}
\newcommand\ABases{\ensuremath{\widehat{\mathcal{B}}}}
\newcommand\Relh[1]{\ensuremath{\mathfrak{h}_{#1}}}
\newcommand\co\colon
\newcommand\BirKer{\ensuremath{\mathcal{K}}}
\newcommand\HIAT{\ensuremath{\widetilde{\HH_2(\IA_n)}}}
\newcommand\Elt[1]{\ensuremath{\|#1\|}}

\makeatletter
\newcommand{\eqnum}{\leavevmode\hfill\refstepcounter{equation}\textup{\tagform@{\theequation}}}
\makeatother

\begin{document}

\maketitle

\vspace{-60pt}
\begin{abstract}
Let $\IA_n$ be the Torelli subgroup of $\Aut(F_n)$.  We give an explicit
finite set of generators for $\HH_2(\IA_n)$ as a $\GL_n(\Z)$-module.  Corollaries
include a version of surjective representation stability for $\HH_2(\IA_n)$,
the vanishing of the $\GL_n(\Z)$-coinvariants of $\HH_2(\IA_n)$, and the vanishing
of the second rational homology group of the level $\ell$ congruence subgroup of
$\Aut(F_n)$.  Our generating set is derived from a new group presentation for $\IA_n$ which is infinite but which has a simple
recursive form.
\end{abstract}

\section{Introduction}

The {\em Torelli subgroup} of the automorphism group of a free group $F_n$ on $n$ letters,
denoted $\IA_n$, is the kernel of the action of $\Aut(F_n)$ on $F_n^{\text{ab}} \cong \Z^n$.
The group of automorphisms of $\Z^n$ is $\GL_n(\Z)$ and the resulting
map $\Aut(F_n) \rightarrow \GL_n(\Z)$ is easily seen to be surjective, so we have a short exact sequence
\[1 \longrightarrow \IA_n \longrightarrow \Aut(F_n) \longrightarrow \GL_n(\Z) \longrightarrow 1.\]
Though it has a large literature, the cohomology and combinatorial group theory 
of $\IA_n$ remain quite mysterious.
Magnus \cite{MagnusGenerators} proved that $\IA_n$ is finitely generated,
and thus that $\HH_1(\IA_n)$ has finite rank.  Krsti\'c--McCool \cite{KrsticMcCool} later
showed that $\IA_3$ is not finitely presentable.  This was improved by
Bestvina--Bux--Margalit \cite{BestvinaBuxMargalitIA}, who showed that $\HH_2(\IA_3)$ has infinite
rank.  However, for $n \geq 4$ it is not known whether or not $\IA_n$ is finitely presentable
or whether or not $\HH_2(\IA_n)$ has finite rank. 

\paragraph{Representation-theoretic finiteness.}
It seems to be very difficult to determine whether or not $\HH_2(\IA_n)$ has finite rank, so
it is natural to investigate weaker sorts of finiteness properties.
Since inner automorphisms act trivially on homology, the conjugation action of $\Aut(F_n)$ on 
$\IA_n$ induces an action of $\GL_n(\Z)$ on $\HH_k(\IA_n)$.  Church--Farb \cite[Conjecture 6.7]{ChurchFarbRepStability}
conjectured that $\HH_k(\IA_n)$ is finitely generated as a $\GL_n(\Z)$-module.  In other words,
they conjectured that there exists a finite subset of $\HH_k(\IA_n)$ whose $\GL_n(\Z)$-orbit
spans $\HH_k(\IA_n)$.  Our first main theorem verifies their conjecture for $k=2$.

\begin{maintheorem}[Generators for $\HH_2(\IA_n)$]
\label{maintheorem:h2finite}
For all $n \geq 2$, there exists a finite subset of $\HH_2(\IA_n)$ whose $\GL_n(\Z)$-orbit spans
$\HH_2(\IA_n)$.
\end{maintheorem}

Each element of our finite subset corresponds to a map of a surface
into a classifying space for $\IA_n$; the genera of these surfaces range from $1$ to $3$.
Table~\ref{table:commutatorrelatorssmall} below lists our finite set of $\GL_n(\Z)$-generators for $\HH_2(\IA_n)$.
This table expresses these generators using specific ``commutator relators'' in $\IA_n$; see below for
how to translate these into elements of $\HH_2(\IA_n)$.

\begin{remark}
The special case $n=3$ of Theorem \ref{maintheorem:h2finite} was proven in the
unpublished thesis of Owen Baker \cite{BakerThesis}.  His proof uses a ``Jacobian''
map on Outer space and is quite different from our proof.
It seems difficult to generalize his proof to higher $n$.
\end{remark}

\paragraph{Surjective representation stability.}
The generators for $\HH_2(\IA_n)$ given in Theorem \ref{maintheorem:h2finite} are explicit enough
that they can be used to perform a number of interesting calculations.  The first verifies
part of a conjecture of Church--Farb that asserts that the homology groups of $\IA_n$ are
``representation stable''.  We begin with some background.  An increasing sequence
\[G_1 \subset G_2 \subset G_3 \subset \cdots\]
of groups is {\em homologically stable} if for all $k \geq 1$, the $k^{\text{th}}$ homology group of $G_n$
is independent of $n$ for $n \gg 0$.  Many sequences of groups are homologically stable; 
see \cite{HatcherWahl} for a bibliography.
In particular, Hatcher--Vogtmann \cite{HatcherVogtmannCerf} proved
this for $\Aut(F_n)$.  However, it is known that $\IA_n$ is {\em not}
homologically stable; indeed, even $\HH_1(\IA_n)$ does not stabilize (see below).

Church--Farb \cite{ChurchFarbRepStability} introduced a new form of homological
stability for groups like $\IA_n$ whose homology groups possess natural group actions.  For
$\IA_n$, they conjectured that for all $k \geq 1$, there exists some $n_k \geq 1$ such
that the following two properties hold for all $n \geq n_k$.
\begin{compactitem}
\item (Injective stability) The map $\HH_k(\IA_n) \rightarrow \HH_k(\IA_{n+1})$ is injective.
\item (Surjective representation stability) The map $\HH_k(\IA_n) \rightarrow \HH_k(\IA_{n+1})$
is surjective ``up to the action of $\GL_{n+1}(\Z)$''; more precisely, the $\GL_{n+1}(\Z)$-orbit
of its image spans $\HH_k(\IA_{n+1})$.
\end{compactitem}
\begin{remark}
In fact, they made this conjecture in \cite{ChurchFarbJacobian} for the Torelli subgroup
of the mapping class group; however, they have informed us that they also conjecture it
for $\IA_n$.
\end{remark}
Our generators for $\HH_2(\IA_n)$ are ``the same in each dimension'' starting at $n=6$, so we are able to
derive the following special case of Church--Farb's conjecture.

\begin{maintheorem}[Surjective representation stability for $\HH_2(\IA_n)$]
\label{maintheorem:surjectivestability}
For $n \geq 6$, the $\GL_{n+1}(\Z)$-orbit of the image of the natural map $\HH_2(\IA_n) \rightarrow \HH_2(\IA_{n+1})$
spans $\HH_2(\IA_{n+1})$.
\end{maintheorem}

\begin{remark}
Boldsen--Dollerup \cite{BoldsenDollerup} proved a theorem similar to Theorem \ref{maintheorem:surjectivestability}
for the {\em rational} second homology group of the Torelli subgroup of the mapping class group.  Their
proof is different from ours; in particular, they were not able to prove an analogue
of Theorem \ref{maintheorem:h2finite}.  It seems hard to use their techniques to prove Theorem
\ref{maintheorem:surjectivestability}.  Similarly, our proof uses special properties of $\IA_n$
and does not work for the Torelli subgroup of the mapping class group.
\end{remark}

\paragraph{Coinvariants.}
Our tools do not allow us to easily distinguish different homology classes; indeed, for all we
know our generators for $\HH_2(\IA_n)$ might be redundant.
This prevents
us from proving injective stability for $\HH_2(\IA_n)$.  However, we still can prove
some interesting vanishing results.  If $G$ is a group and $M$ is a $G$-module, then
the {\em coinvariants} of $G$ acting on $M$, denoted $M_G$, are the largest quotient
of $M$ on which $G$ acts trivially.  More precisely, $M_G = M/K$ with
$K = \langle \text{$m - g \cdot m$ $|$ $g \in G$, $m \in M$} \rangle$.  We then have
the following.

\begin{maintheorem}[Vanishing coinvariants]
\label{maintheorem:coinvariants}
For $n \geq 6$, we have $(\HH_2(\IA_n))_{\GL_n(\Z)} = 0$.
\end{maintheorem}

\begin{remark}
In \cite[Conjecture 6.5]{ChurchFarbRepStability}, Church--Farb conjectured that the
$\GL_n(\Z)$-invariants in $\HH^k(\IA_n;\Q)$ are $0$.  For $k = 1$, this follows
from the known computation of $\HH^1(\IA_n;\Q)$; see below.  Theorem \ref{maintheorem:coinvariants}
implies that this also holds for $k=2$.
\end{remark}

\paragraph{Linear congruence subgroups.}
For $\ell \geq 2$, the {\em level $\ell$ congruence subgroup} of $\Aut(F_n)$, denoted $\Aut(F_n,\ell)$,
is the kernel of the natural map $\Aut(F_n) \rightarrow \GL_n(\Z/\ell)$; one should think of it
as a ``mod-$\ell$'' version of $\IA_n$.  It is natural to conjecture that for all $k \geq 1$,
there exists some $n_k \geq 1$ such that $\HH_k(\Aut(F_n,\ell);\Q) \cong \HH_k(\Aut(F_n);\Q)$ for
$n \geq n_k$; an analogous theorem for congruence subgroups of $\GL_n(\Z)$ is due
to Borel \cite{BorelStability1}.  Galatius \cite{GalatiusAut} proved that $\HH_k(\Aut(F_n);\Q) = 0$
for $n \gg 0$, so this conjecture really asserts that $\HH_k(\Aut(F_n,\ell);\Q) = 0$ for $n \gg 0$.
The case $k=1$ of this is known.  Indeed, Satoh \cite{SatohCongruence} calculated
the abelianization of $\Aut(F_n,\ell)$ for $n \geq 3$ 
and the answer consisted entirely of torsion, so $\HH_1(\Aut(F_n,\ell);\Q)=0$
for $n \geq 3$.  Using Theorem \ref{maintheorem:h2finite}, we will prove the case $k=2$.

\begin{maintheorem}[Second homology of congruence subgroups]
\label{maintheorem:congruence}
For $\ell \geq 2$ and $n \geq 6$, we have $\HH_2(\Aut(F_n,\ell);\Q) = 0$.
\end{maintheorem}
The key to our proof is that Theorem \ref{maintheorem:h2finite} allows us to show that
the image of $\HH_2(\IA_n;\Q)$ in $\HH_2(\Aut(F_n,\ell);\Q)$ vanishes; this allows us
to derive Theorem \ref{maintheorem:congruence} using standard techniques.

\begin{remark}
The second author proved an analogue of Theorem \ref{maintheorem:congruence} for congruence
subgroups of the mapping class group in \cite{PutmanH2Congruence}.  The techniques
in \cite{PutmanH2Congruence} are different from those in the present paper and
it seems difficult to prove Theorem \ref{maintheorem:congruence} via those techniques.
\end{remark}

\paragraph{Basic elements of Torelli.}
We now wish to describe our generating set for $\HH_2(\IA_n)$.  This requires
introducing some basic elements of $\IA_n$.  Let $\{x_1,\ldots,x_n\}$ be a
free basis for $F_n$.  We then make the following definitions.
\begin{compactitem}
\item For distinct $1 \leq i,j \leq n$, let $\Con{x_i}{x_j} \in \IA_n$ be defined via the formulas
\[\Con{x_i}{x_j}(x_i) = x_j x_i x_j^{-1} \quad \text{and} \quad \Con{x_i}{x_j}(x_{\ell}) = x_{\ell} \quad \text{if $\ell \neq i$}.\]
\item For $\alpha,\beta,\gamma \in \{\pm 1\}$ and distinct $1 \leq i,j,k \leq n$, 
let $\Mulcomm{x_i^{\alpha}}{x_j^{\beta}}{x_k^{\gamma}} \in \IA_n$ be defined via the formulas
\[\Mulcomm{x_i^{\alpha}}{x_j^{\beta}}{x_k^{\gamma}}(x_i^{\alpha}) = [x_j^{\beta},x_k^{\gamma}] x_i^{\alpha} \quad \text{and}
\quad \Mulcomm{x_i^{\alpha}}{x_j^{\beta}}{x_k^{\gamma}}(x_{\ell}) = x_{\ell} \quad \text{if $\ell \neq i$}.\]
\end{compactitem}
Observe that by definition
\[M_{x_i^{-1},[x_j^{\beta},x_k^{\gamma}]}(x_i^{-1}) = [x_j^{\beta},x_k^{\gamma}] x_i^{-1} \quad \text{and} \quad M_{x_i^{-1},[x_j^{\beta},x_k^{\gamma}]}(x_i) = x_i [x_j^{\beta},x_k^{\gamma}]^{-1}.\]
We call $\Con{x_i}{x_j}$ a {\em conjugation move} and $\Mulcomm{x_i^{\alpha}}{x_j^{\beta}}{x_k^{\gamma}}$ 
a {\em commutator transvection}.

\paragraph{Surfaces in a classifying space: our generators.}
A {\em commutator relator} in $\IA_n$ is a formula of the form $[a_1,b_1] \cdots [a_g,b_g] = 1$ with $a_i,b_i \in \IA_n$.
Given such a commutator relator $r$, let $\Sigma_g$ be a genus $g$ surface.  There is 
a continuous map $\zeta \co \Sigma_g \rightarrow K(\IA_n,1)$ that takes the standard basis for
$\pi_1(\Sigma_g)$ to $a_1,b_1,\ldots,a_g,b_g \in \IA_n$.
We obtain an element $\Relh{r} = \zeta_{\ast}([\Sigma_g]) \in \HH_2(\IA_n)$.
With this notation, the generators for $\HH_2(\IA_n)$ given by Theorem \ref{maintheorem:h2finite} are the elements
$\Relh{r}$ where $r$ is one of the relators in Table \ref{table:commutatorrelatorssmall}.

%%%%%%%%%%%%%%%%%%%%%%%%%%%%%%%%%%%%%%%%%%%%%%%%%%%%%%%%%%%%%%%%%%%%%%%%%%%%%%%%%%%
%%%%%%%%%%%%%%%%%%%%%%%%%%%%%%%%%%%%%%%%%%%%%%%%%%%%%%%%%%%%%%%%%%%%%%%%%%%%%%%%%%%
%%%%%%%%%%%%%%%%%%%%%%%%%%%%%%%%%%%%%%%%%%%%%%%%%%%%%%%%%%%%%%%%%%%%%%%%%%%%%%%%%%%
\begin{table}[t!]
\begin{tabular}{p{0.95\textwidth}}
\toprule
\vspace{-5pt}\begin{enumerate}\setlength{\itemsep}{0ex}
\item[H1.] $[\Con{x_a}{x_b},\Con{x_c}{x_d}]=1$, possibly with $b=d$.
\item[H2.] $[\Mulcomm{x_a^\alpha}{x_b^\beta}{x_c^\gamma},\Mulcomm{x_d^\delta}{x_e^\epsilon}{x_f^\zeta}]=1$, 
possibly with $\{b,c\}\cap\{e,f\}\neq \varnothing$ or with $x_a^{\alpha}=x_d^{-\delta}$, as long as $x_a^\alpha\neq x_d^\delta$, $a\notin\{e,f\}$ and $d\notin \{b,c\}$.
\item[H3.] $[\Con{x_a}{x_b},\Mulcomm{x_c^\gamma}{x_d^\delta}{x_e^\epsilon}]=1$, possibly with $b\in\{d,e\}$, if $c\notin\{a,b\}$ and $a\notin\{c,d,e\}$.
\item[H4.] $[\Con{x_c}{x_b}^\beta\Con{x_a}{x_b}^\beta,\Con{x_c}{x_a}^\alpha]=1$.
\item[H5.] $[\Con{x_a}{x_c}^{-\gamma},\Con{x_a}{x_d}^{-\delta}][\Con{x_a}{x_b}^{-\beta},\Mulcomm{x_b^\beta}{x_c^\gamma}{x_d^\delta}]=1$.
\item[H6.] $[\Mulcomm{x_a^\alpha}{x_b^\beta}{x_c^\gamma},\Mulcomm{x_d^\delta}{x_a^\alpha}{x_e^\epsilon}]
[\Mulcomm{x_d^\delta}{x_a^\alpha}{x_e^\epsilon},\Mulcomm{x_d^\delta}{x_c^\gamma}{x_b^\beta}]
[\Mulcomm{x_d^\delta}{x_c^\gamma}{x_b^\beta},\Con{x_d}{x_e}^{-\epsilon}]
=1$, possibly with $b=e$ or $c=e$.
\item[H7.] $[\Mulcomm{x_c^\gamma}{x_a^\alpha}{x_d^\delta},\Con{x_a}{x_b}^\beta]
[\Con{x_c}{x_d}^{-\delta},\Mulcomm{x_c^\gamma}{x_a^\alpha}{x_b^\beta}]
[\Mulcomm{x_c^\gamma}{x_a^\alpha}{x_b^\beta},\Mulcomm{x_c^\gamma}{x_a^\alpha}{x_d^\delta}]
=1$, 
possibly with $b=d$.
\item[H8.] $[\Mulcomm{x_a^\alpha}{x_b^\beta}{x_c^\gamma},\Con{x_a}{x_d}^\delta\Con{x_b}{x_d}^\delta\Con{x_c}{x_d}^\delta]=1$.
\item[H9.] $[\Con{x_a}{x_c}^\gamma\Con{x_b}{x_c}^\gamma,\Con{x_a}{x_b}^\beta\Con{x_c}{x_b}^\beta]
[\Mulcomm{x_a^\alpha}{x_b^\beta}{x_c^\gamma},\Con{x_b}{x_a}^{\alpha}\Con{x_c}{x_a}^\alpha] = 1$.
\end{enumerate}\\
\bottomrule
\end{tabular}
\caption{
The set of commutator relators whose associated elements of $\HH_2(\IA_n)$ generate
it as a $\GL_n(\Z)$-module.
Distinct letters represent distinct indices unless stated otherwise.
}
\label{table:commutatorrelatorssmall}
\end{table}

%%%%%%%%%%%%%%%%%%%%%%%%%%%%%%%%%%%%%%%%%%%%%%%%%%%%%%%%%%%%%%%%%%%%%%%%%%%%%%%%%%%
%%%%%%%%%%%%%%%%%%%%%%%%%%%%%%%%%%%%%%%%%%%%%%%%%%%%%%%%%%%%%%%%%%%%%%%%%%%%%%%%%%%
%%%%%%%%%%%%%%%%%%%%%%%%%%%%%%%%%%%%%%%%%%%%%%%%%%%%%%%%%%%%%%%%%%%%%%%%%%%%%%%%%%%

\paragraph{The Johnson homomorphism.}
To motivate our proof of Theorem \ref{maintheorem:h2finite}, we must first recall
the computation of $\HH_1(\IA_n)$, which is due independently to
Farb \cite{FarbIA}, Kawazumi \cite{KawazumiMagnus}, and Cohen--Pakianathan \cite{CohenPakianathan}.
The basic tool is the Johnson homomorphism, which was introduced by Johnson \cite{JohnsonHomo} in the context
of the Torelli subgroup of the mapping class group (though it also appears in
earlier work of Andreadakis \cite{Andreadakis}).  See \cite{SatohSurvey} for a survey
of the $\IA_n$-version of it.  The Johnson homomorphism is a homomorphism
\[\tau\colon \IA_n \rightarrow \Hom(\Z^n,\textstyle{\bigwedge^2} \Z^n)\]
that arises from studying the action of $\IA_n$ on the second nilpotent truncation
of $F_n$.  It can be defined as follows.  For $z \in F_n$, let $[z] \in \Z^n$ be the associated
element of the abelianization of $F_n$.  Consider $f \in \IA_n$.  For $x \in F_n$,
we have $f(x) \cdot x^{-1} \in [F_n,F_n]$.  There is a natural surjection $\rho\colon [F_n,F_n] \rightarrow \bigwedge^2 \Z^n$
satisfying $\rho([a,b]) = [a] \wedge [b]$; the kernel of $\rho$ is $[F_n,[F_n,F_n]]$.  We can then
define a map $\tilde{\tau}_f \colon F_n \rightarrow \bigwedge^2 \Z^n$ via the formula $\tilde{\tau}_f(x) = \rho(f(x) \cdot x^{-1})$.
One can check that $\tilde{\tau}_f$ is a homomorphism.  It factors through a homomorphism
$\tau_f \colon \Z^n \rightarrow \bigwedge^2 \Z^n$.  We can then define $\tau \colon \IA_n \rightarrow \Hom(\Z^n,\bigwedge^2 \Z^n)$
via the formula $\tau(f) = \tau_f$.  One can check that $\tau$ is a homomorphism.

\paragraph{Generators and their images.}
Define
\[S_{\Mag}(n) = \Set{$\Con{x_i}{x_j}$}{$1 \leq i,j \leq n$ distinct } \cup \Set{$\Mulcomm{x_i}{x_j}{x_k}$}{$1 \leq i,j,k \leq n$ distinct, $j < k$}.\]
Magnus \cite{MagnusGenerators} proved that $\IA_n$ is generated by $S_{\Mag}(n)$; see 
\cite{DayPutmanComplex} and \cite{BestvinaBuxMargalitIA} for modern proofs.
For distinct $1 \leq i,j \leq n$, the image $\tau(\Con{x_i}{x_j}) \in \Hom(\Z^n,\bigwedge^2 \Z^n)$ is the homomorphism
defined via the formulas
\[[x_i] \mapsto [x_j] \wedge [x_i] \quad \text{and} \quad [x_{\ell}] \mapsto 0 \quad \text{if $\ell \neq i$}.\]
Similarly, for distinct $1 \leq i,j,k \leq n$ with $j < k$, the image $\tau(\Mulcomm{x_i}{x_j}{x_k}) \in \Hom(\Z^n,\bigwedge^2 \Z^n)$ is the homomorphism
defined via the formulas
\[[x_i] \mapsto [x_j] \wedge [x_k] \quad \text{and} \quad [x_{\ell}] \mapsto 0 \quad \text{if $\ell \neq i$}.\]
The key observation is that these form a {\em basis} for $\Hom(\Z^n,\bigwedge^2 \Z^n)$.

\paragraph{The abelianization.}
Let $F(S_{\Mag}(n))$ be the free group on $S_{\Mag}(n)$ and let $R_{\Mag}(n) \subset F(S_{\Mag}(n))$ be a
set of relations for $\IA_n$, so $\IA_n = \Pres{S_{\Mag}(n)}{R_{\Mag}(n)}$.  Since $\tau$
takes $S_{\Mag}(n)$ bijectively to a basis for the free abelian group $\Hom(\Z^n,\bigwedge^2 \Z^n)$, we
must have $R_{\Mag}(n) \subset [F(S_{\Mag}(n)),F(S_{\Mag}(n))]$.  This immediately implies that
$\HH_1(\IA_n) \cong \Hom(\Z^n,\bigwedge^2 \Z^n)$.  

\paragraph{Hopf's formula.}
But even more is true.  Recall that Hopf's formula \cite{BrownCohomology} says that if $G$ is a group
with a presentation $G = \Pres{S}{R}$, then
\[\HH_2(G) \cong \frac{\llangle R \rrangle \cap [F(S),F(S)]}{[F(S),\llangle R \rrangle]};\]
here $\llangle R \rrangle$ is the normal closure of $R$.  The intersection in the numerator of this
is usually hard to calculate, so Hopf's formula is not often useful for computation.  However, by
what we have said it simplifies for $\IA_n$ to
\begin{equation}
\label{eqn:h2iahopf}
\HH_2(\IA_n) \cong \frac{\llangle R_{\Mag}(n) \rrangle}{[F(S_{\Mag}(n),\llangle R_{\Mag}(n) \rrangle]}.
\end{equation}
This isomorphism is very concrete: an element $r \in \llangle R_{\Mag}(n) \rrangle$ is a commutator
relator, and the associated element of $\HH_2(\IA_n)$ is the homology class $\Relh{r}$ discussed above.

\paragraph{Summary and trouble.}
For $r \in \llangle R_{\Mag}(n) \rrangle$ and $z \in F(S_{\Mag}(n))$, the element
$z r z^{-1} r^{-1}$ lies in the denominator of \eqref{eqn:h2iahopf}, i.e.\ $[F(S_{\Mag}(n)),\llangle R_{\Mag}(n) \rrangle]$.
Hence $\Relh{z r z^{-1}} = \Relh{r}$.  It follows that $\HH_2(\IA_n)$ is generated by the set
$\Set{$\Relh{r}$}{$r \in R_{\Mag}(n)$}$.  In other words, to calculate generators for $\HH_2(\IA_n)$, it is enough
to find a presentation for $\IA_n$ with $S_{\Mag}(n)$ as its generating set.  However, this
seems like a difficult problem (especially if, as we suspect, $\IA_n$ is not
finitely presentable).  Moreover, the $\GL_n(\Z)$-action on $\HH_2(\IA_n)$ has not
yet appeared.

\paragraph{L-presentations.}
To incorporate the $\GL_n(\Z)$-action on $\HH_2(\IA_n)$ into our presentation for $\IA_n$,
we  use the notion of an L-presentation, which was introduced by Bartholdi \cite{BartholdiBranch} (we use a slight simplification of his definition).
An {\em L-presentation} for a group $G$ is a triple $\LPres{S}{R^0}{E}$,
where $S$ and $R^0$ and $E$ are as follows.
\begin{compactitem}
\item $S$ is a generating set for $G$.
\item $R^0 \subset F(S)$ is a set consisting of relations for $G$ (not necessarily complete).
\item $E$ is a subset of $\End(F(S))$.
\end{compactitem}
This data must satisfy the following condition.  Let $M \subset \End(F(S))$ be the monoid
generated by $E$.  Define $R = \Set{$f(r)$}{$f \in M$, $r \in R^0$}$.  Then we require
that $G = \Pres{S}{R}$.
Each element of $E$ descends to an element of $\End(G)$; we call the resulting subset 
$\widetilde{E} \subset \End(G)$ the {\em induced endomorphisms} of our L-presentation.  We
 say that our L-presentation is {\em finite} if the sets $S$ and $R^0$ and $E$ are all finite.

In our examples, the induced endomorphisms of our L-presentations will actually be automorphisms.
Thus in the context of this paper one should think of an L-presentation as a group presentation
incorporating certain symmetries of a group.  Here is an easy example.

\begin{example}
Let $S = \Set{$z_i$}{$i \in \Z/p$}$ and $R^0 = \{z_0^2\}$.  Let $\psi \colon F(S) \rightarrow F(S)$ be the
homomorphism defined via the formula $\psi(z_i) = z_{i+1}$.  Then $\LPres{S}{R^0}{\{\psi\}}$ is an L-presentation
for the free product of $p$ copies of $\Z/2$.
\end{example}

\paragraph{A finite L-presentation for Torelli.}
The conjugation action of $\Aut(F_n)$ on $\IA_n$ gives an injection $\Aut(F_n) \hookrightarrow \Aut(\IA_n)$.
If we could somehow construct a finite L-presentation $\LPres{S_{\Mag}(n)}{R_{\Mag}^0(n)}{E_{\Mag}(n)}$
for $\IA_n$ whose set of induced endomorphisms generated 
\[\Aut(F_n) \subset \Aut(\IA_n) \subset \End(\IA_n),\]
then Theorem \ref{maintheorem:h2finite} would immediately follow.  Indeed, 
since the $\GL_n(\Z)$-action on $\HH_2(\IA_n)$ is induced by the conjugation action
of $\Aut(F_n)$ on $\IA_n$, it would follow that the $\GL_n(\Z)$-orbit of the set
$\Set{$\Relh{r}$}{$r \in R^0_{\Mag}(n)$} \subset \HH_2(\IA_n)$ spanned $\HH_2(\IA_n)$.

Although we find the idea in the previous paragraph illuminating, we do not follow it strictly.
To make our L-presentation for $\IA_n$ easier to comprehend, we will use the following generating
set, which is larger than $S_{\Mag}$:
\begin{align*}
S_{\IA}(n) =\ &\Set{$\Con{x_i}{x_j}$}{$1 \leq i,j \leq n$ distinct} \\
&\quad\quad\cup \Set{$\Mulcomm{x_i^{\alpha}}{x_j^{\beta}}{x_k^{\gamma}}$}{$1 \leq i,j,k \leq n$ distinct, $\alpha,\beta,\gamma \in \{\pm 1\}$}.
\end{align*}
This has the advantage of making our relations and rewriting rules shorter, and making their meaning easier to understand.
It has the disadvantage of making the proof of Theorem~\ref{maintheorem:h2finite} less direct.
Our theorem giving an L-presentation for $\IA_n$ is as follows.

\begin{maintheorem}[Finite L-presentation for Torelli.]
\label{maintheorem:finitelpres}
For all $n \geq 2$, there exists a finite L-presentation
$\IA_n = \LPres{S_{\IA}(n)}{R_{\IA}^0(n)}{E_{\IA}(n)}$
whose set of induced endomorphisms generates $\Aut(F_n) \subset \Aut(\IA_n) \subset \End(\IA_n)$.
\end{maintheorem}

We note that our presentation is not a presentation in which all relators are commutators.
The formulas for the $R^0_{\IA}(n)$ and $E_{\IA}(n)$ in our finite L-presentation are a little complicated, so we postpone
them until \S \ref{section:thepresentation}.
The formulas in that section make it clear that $R_{\IA}(n)$ does not lie in $[F(S_{\IA}(n)),F(S_{\IA}(n))]$.
Therefore we cannot prove Theorem~\ref{maintheorem:h2finite} simply by interepreting the relators as homology classes.
We must do something more complicated to deduce that theorem from our presentation.

\begin{remark}
The relations in Table \ref{table:commutatorrelatorssmall} are not sufficient for our L-presentation.  Indeed,
they all lie in the commutator subgroup, but the generators $S_{\IA}(n)$ do not map to linearly independent
elements of the abelianization.
\end{remark}

\paragraph{Sketch of proof.}
We close this introduction by briefly discussing how we prove Theorem \ref{maintheorem:finitelpres}.
In particular, we explain why it is easier to
verify an L-presentation than a standard presentation.  We remark that our proof
is inspired by a recent paper \cite{BrendleMargalitPutman} of the second author together with Brendle and Margalit
which constructed generators for the kernel of the Burau representation evaluated at $-1$.

Assume that we have guessed a finite L-presentation $\LPres{S_{\IA}(n)}{R^0_{\IA}(n)}{E_{\IA}(n)}$
for $\IA_n$ as in Theorem \ref{maintheorem:finitelpres} (we found the one that we use
by first throwing in all the relations we could think of and then attempting the proof
below; each time it failed it revealed a relation that we had missed).  Let $\Quotient_n$ be the group 
presented by the purported L-presentation.  There
is thus a surjection $\pi \colon \Quotient_n \rightarrow \IA_n$, and the goal of our proof will be to construct
an inverse map $\phi \colon \IA_n \rightarrow \Quotient_n$ satisfying $\phi \circ \pi = \text{id}$.  This will
involve several steps.

\BeginSteps
\begin{step}
We decompose $\IA_n$ in terms of stabilizers of conjugacy classes of primitive elements of $F_n$.
\end{step}
\vspace{-5pt}
For $z \in F_n$, let $\Conj{z}$ denote the union of the conjugacy classes of $z$ and $z^{-1}$.
A {\em primitive element} of $F_n$ is an element that forms part of a free basis.  Let
$\mathcal{C} = \Set{$\Conj{z}$}{$z \in F_n$ primitive}$.  The 
set $\mathcal{C}$ forms the set of vertices of a simplicial complex called the {\em complex
of partial bases} which is analogous to the complex of curves for the mapping class
group.  Applying a theorem of the second author \cite{PutmanPresentation} to the action
of $\IA_n$ on the complex of partial bases, we will obtain a decomposition
\begin{equation}
\label{eqn:decomposeia}
\IA_n = \BigFreeProd_{c \in \mathcal{C}} (\IA_n)_c / (\text{some relations});
\end{equation}
here $(\IA_n)_c$ denotes the stabilizer in $\IA_n$ of $c$.  The unlisted relations play
only a small role in our proof and can be ignored at this point.

\begin{step}
We use induction to construct a partial inverse.
\end{step}
\vspace{-5pt}
Fix some $c_0 \in \mathcal{C}$.  The stabilizer
$(\IA_n)_{c_0}$ is very similar to $\IA_{n-1}$; in fact, it is connected to $\IA_{n-1}$
by an exact sequence that is analogous to the Birman exact sequence for the
mapping class group.  We construct this exact sequence in the companion paper
\cite{DayPutmanBirmanIA}, which builds on our previous paper \cite{DayPutmanBirman}.
By analyzing this exact sequence and using induction, we will
construct a ``partial inverse'' $\phi_{c_0} \colon (\IA_n)_{c_0} \rightarrow \Quotient_n$.  We remark
that this step is where most of our relations arise---they actually are relations in the kernel
of the Birman exact sequence we construct in \cite{DayPutmanBirman}.

\begin{step}
We use the L-presentation to lift the conjugation action of $\Aut(F_n)$ on $\IA_n$ to $\Quotient_n$.
\end{step}
\vspace{-5pt}
Let $\widetilde{E}$ be the induced endomorphisms of our L-presentation.  We directly
prove that these endomorphisms actually give an action of $\Aut(F_n)$ on $\Quotient_n$
such that the projection map $\pi \colon \Quotient_n \rightarrow \IA_n$ is equivariant.  This is the key place where
we use properties of L-presentations; in general, it is difficult to construct
group actions on groups given by generators and relations.

\begin{step}
We use our group action to construct the inverse.
\end{step}
\vspace{-5pt}
The conjugation action of $\Aut(F_n)$ on $\IA_n$ transitively permutes the terms of
\eqref{eqn:decomposeia}.  Using our lifted action of $\Aut(F_n)$ on $\Quotient_n$ as a ``guide'', we
then ``move'' the partially defined inverse $\phi_{c_0}$ around and construct $\phi$ on the rest
of $\IA_n$, completing the proof.

\begin{remark}
In \cite{PutmanInfinite}, the second author constructed an infinite presentation of the
Torelli subgroup of the mapping class group.  Though this used the same result \cite{PutmanPresentation}
that we quoted above, the details are quite different.  One source of this difference
is that instead of an L-presentation with a finite generating set, the paper \cite{PutmanInfinite}
constructed an ordinary presentation with an infinite generating set.  In
our paper \cite{DayPutmanBasisFree}, we construct a different presentation
for $\IA_n$ which is in the same spirit as the presentation in \cite{PutmanInfinite}.
\end{remark}

\paragraph{Computer calculations.}
At several places in this paper, we will need to verify large numbers of equations in group presentations.
Rather than displaying these equations in the paper or leaving them as exercises, we use the GAP System to store and check our equations mechanically.
The code to verify these equations is in the file \verb+h2ia.g+, which
is distributed with this document and is also available on the authors' websites.  It is also attached
to this paper's arXiv posting.
We found the equations in this file by hand, and our proof does not rely on a computer search.
We will say more about this in \S \ref{section:somecomputations}, where said
calculations begin. 

This is a good place to note that our results rely strongly on the authors' earlier paper~\cite{DayPutmanBirmanIA} and the computer calculutions from that paper.
In that paper, we use a similar approach to automatically verify identities and prove the existence of certain homomorphisms to between groups given by presentations.
This is used in more than one place in the present paper, but most crucially in Proposition~\ref{proposition:mapfromkernel}.
The computations from our earlier paper are in a file \verb+iabes.g+, which is also available on the authors' websites and
on the arXiv.

\paragraph{Outline.}
We begin in \S \ref{section:thepresentation} by giving a precise statement of
the L-presentation whose existence
is asserted in Theorem \ref{maintheorem:finitelpres}.
Next, in \S \ref{section:prooftools} we discuss several tools that are needed for the
proof of Theorem \ref{maintheorem:finitelpres}.  
The proof of Theorem \ref{maintheorem:finitelpres}
is in \S \ref{section:theproof}.  This proof depends on some combinatorial group theory
calculations that are stated in \S \ref{section:thepresentation} and \S \ref{section:prooftools} but 
whose proofs are postponed until \S \ref{section:somecomputations}.  In \S \ref{section:h2gen},
we prove Theorem \ref{maintheorem:h2finite}.  That section also shows how to derive
Theorem \ref{maintheorem:surjectivestability} from Theorem \ref{maintheorem:h2finite}.
Finally, Theorems \ref{maintheorem:coinvariants} and
Theorem \ref{maintheorem:congruence} are proven in \S \ref{section:finalcalcs}.

\section{Our finite L-presentation}
\label{section:thepresentation}

We now discuss the relations $R_{\IA}^0(n)$ and endomorphisms $E_{\IA}(n)$ of our L-presentation.
Two calculations (Propositions \ref{proposition:verifyrels} and \ref{proposition:verifyendomorphisms})
are postponed until \S \ref{section:somecomputations}.

\paragraph{Relations.}
Our set $R_{\IA}^0(n)$ of relations consists of the relations in Table \ref{table:basicrelsian}.  
It is easy to verify that these are all indeed relations:
\begin{proposition}\label{proposition:verifyrels}
The relations $R_{\IA}^0(n)$ all hold when interpreted in $\IA_n$.
\end{proposition}
The proof is computational and is postponed until \S \ref{section:somecomputations}.
We remark that unlike many of our computer calculations, it is not particularly difficult to verify by hand.

Since the relations are rather
complicated we suggest to the reader that they not pay too close attention to them
on their first pass through the paper.  
The overall structure of our proof (and, in fact, the majority of its details) can be understood without much knowledge of our relations.

\begin{remark}
The relations in $R_{\IA}^0(n)$ have reasonable intuitive interpretations.
R1 through R3 state that generators acting only in different places commute with each other.
R4 is a generalization of the fact that for $n=3$, the conjugation move $\Con{x_3}{x_1}$ conjugates 
the inner automorphism $\Con{x_1}{x_2}\Con{x_3}{x_2}$ back to itself (since it fixes the conjugating element $x_2$).
R5 makes sense by looking at either side of $x_a$: on the right of $x_a^{\alpha}$, instances of 
$x_b^{\pm\beta}$ cancel, but on the left side of $x_a^{\alpha}$, we get a conjugate of a 
basic commutator that is itself a basic commutator.
R6 states that conjugation by a commutator is the same as acting by a commutator of conjugation moves.
R7, R8 and R9 allow us to rewrite a conjugate of a generator acting on a given element as a product of 
generators acting only on that same element ($x_a$, $x_d$ or $x_c$ as stated here, respectively).
In this sense, these relations are like the Steinberg relations from the presentation of $GL_n(\Z)$ in algebraic K-theory.
\end{remark}

%%%%%%%%%%%%%%%%%%%%%%%%%%%%%%%%%%%%%%%%%%%%%%%%%%%%%%%%%%%%%%%%%%%%%%%%%%%%%%%%%%%
%%%%%%%%%%%%%%%%%%%%%%%%%%%%%%%%%%%%%%%%%%%%%%%%%%%%%%%%%%%%%%%%%%%%%%%%%%%%%%%%%%%
%%%%%%%%%%%%%%%%%%%%%%%%%%%%%%%%%%%%%%%%%%%%%%%%%%%%%%%%%%%%%%%%%%%%%%%%%%%%%%%%%%%
\begin{table}[t!]
\begin{tabular}{p{0.95\textwidth}}
\toprule
\begin{center}
\textbf{Basic relations for $\IA_n$}
\end{center}
Distinct letters are assumed to represent distinct indices unless stated otherwise.
Let $R_{\IA}(n)$ denote the finite set of all relations from the following ten classes.
\vspace*{-0ex}\begin{enumerate}\setlength{\itemsep}{0ex}
\item[R0.]
\(\Mulcomm{x_a^\alpha}{x_b^\beta}{x_c^\gamma}^{-1}=\Mulcomm{x_a^\alpha}{x_c^\gamma}{x_b^\beta}\).
\item[R1.]
\( [\Con{x_a}{x_b},\Con{x_c}{x_d}]=1\),
possibly with $b=d$.
\item[R2.]
\( [\Mulcomm{x_a^\alpha}{x_b^\beta}{x_c^\gamma},\Mulcomm{x_d^\delta}{x_e^\epsilon}{x_f^\zeta}]=1\),
possibly with $\{b,c\}\cap\{e,f\}\neq \varnothing$ or with $x_a^{\alpha}=x_d^{-\delta}$,
as long as $x_a^\alpha\neq x_d^\delta$, $a\notin\{e,f\}$ and $d\notin \{b,c\}$.
\item[R3.]
\( [\Con{x_a}{x_b},\Mulcomm{x_c^\gamma}{x_d^\delta}{x_e^\epsilon}]=1\),
possibly with $b\in\{d,e\}$, if $c\notin\{a,b\}$ and $a\notin\{c,d,e\}$.
\item[R4.]
\([\Con{x_a}{x_b}\Con{x_c}{x_b},\Con{x_c}{x_a}]=1\).
\item[R5.]
\(\Con{x_a}{x_b}^\beta\Mulcomm{x_a^\alpha}{x_b^\beta}{x_c^\gamma}\Con{x_a}{x_b}^{-\beta}=\Mulcomm{x_a^\alpha}{x_c^\gamma}{x_b^{-\beta}}\).
\item[R6.]
\(\Mulcomm{x_a^{\alpha}}{x_b^\beta}{x_c^\gamma}\Mulcomm{x_a^{-\alpha}}{x_b^\beta}{x_c^\gamma}=[\Con{x_a}{x_c}^{-\gamma},\Con{x_a}{x_b}^{-\beta}]\).
\item[R7.]
\([\Con{x_a}{x_b}^{-\beta},\Mulcomm{x_b^\beta}{x_c^\gamma}{x_d^\delta}]=[\Con{x_a}{x_d}^{-\delta},\Con{x_a}{x_c}^{-\gamma}]\).
\item[R8.]
\(
\Mulcomm{x_a^\alpha}{x_b^\beta}{x_c^\gamma}\Mulcomm{x_d^\delta}{x_a^\alpha}{x_e^\epsilon}\Mulcomm{x_a^\alpha}{x_c^\gamma}{x_b^\beta}
=
\Con{x_d}{x_e}^{-\epsilon}\Mulcomm{x_d^\delta}{x_c^\gamma}{x_b^\beta}\Con{x_d^\delta}{x_e}^\epsilon\Mulcomm{x_d^\delta}{x_a^\alpha}{x_e^\epsilon}\Mulcomm{x_d^\delta}{x_b^\beta}{x_c^\gamma},\hfill
\)
possibly with $b=e$ or $c=e$.
\item[R9.]
\(
\Con{x_a}{x_b}^\beta
\Mulcomm{x_c^\gamma}{x_a^\alpha}{x_d^\delta}\Con{x_a}{x_b}^{-^\beta}
=\Con{x_c}{x_d}^{-\delta}\Mulcomm{x_c^\gamma}{x_a^\alpha}{x_b^\beta}\Con{x_c}{x_d}^\delta\Mulcomm{x_c^\gamma}{x_a^\alpha}{x_d^\delta}\Mulcomm{x_c^\gamma}{x_b^\beta}{x_a^\alpha},
\)
possibly with $b=d$.
\end{enumerate}\\
\bottomrule
\end{tabular}
\caption{Relations for the L-presentation of $\IA_n$.}
\label{table:basicrelsian}
\end{table}
%%%%%%%%%%%%%%%%%%%%%%%%%%%%%%%%%%%%%%%%%%%%%%%%%%%%%%%%%%%%%%%%%%%%%%%%%%%%%%%%%%%
%%%%%%%%%%%%%%%%%%%%%%%%%%%%%%%%%%%%%%%%%%%%%%%%%%%%%%%%%%%%%%%%%%%%%%%%%%%%%%%%%%%
%%%%%%%%%%%%%%%%%%%%%%%%%%%%%%%%%%%%%%%%%%%%%%%%%%%%%%%%%%%%%%%%%%%%%%%%%%%%%%%%%%%

\paragraph{Generators for automorphism group of free group.}
Before discussing our endomorphisms $E_{\IA}(n)$, we first introduce 
a generating set for $\Aut(F_n)$ that goes back to work of Nielsen.
For $\alpha = \pm 1$ and distinct $1 \leq i,j \leq n$, let $\Mul{x_i^{\alpha}}{x_j} \in \Aut(F_n)$
be the {\em transvection} that takes $x_i^{\alpha}$ to $x_j x_i^{\alpha}$ and fixes
$x_{\ell}$ for $\ell \neq i$.  Just like before, we have
\[\Mul{x_i^{-1}}{x_j}(x_i^{-1}) = x_j x_i^{-1} \quad \text{and} \quad \Mul{x_i^{-1}}{x_j}(x_i) = x_i x_j^{-1}.\]
Next, for distinct $1 \leq i,j \leq n$ let $P_{i,j} \in \Aut(F_n)$ be the {\em swap automorphism}
that exchanges $x_i$ and $x_j$ while fixing $x_{\ell}$ for $\ell \neq i,j$.  Finally, for $1 \leq i \leq n$
let $I_i \in \Aut(F_n)$ be the {\em inversion automorphism} that takes $x_i$ to $x_i^{-1}$ and fixes
$x_{\ell}$ for $\ell \neq i$.  Define
\begin{align*}
S_{\Aut}(n) = &\Set{$\Mul{x_i^{\alpha}}{x_j}^{\beta}$}{$1 \leq i,j \leq n$ distinct, $\alpha,\beta \in \{\pm 1\}$}\\
&\quad\quad \cup
\Set{$P_{i,j}$}{$1 \leq i,j \leq n$ distinct} \cup
\Set{$I_i$}{$1 \leq i \leq n$}.
\end{align*}
Observe that the set $S_{\Aut}(n) \subset \Aut(F_n)$ is closed under inversion.

\paragraph{Endomorphisms.}
Below we will define a function
$\ttheta \colon S_{\Aut}(n) \rightarrow \End(F(S_{\IA}(n)))$ with the following key
property.  Let $\pi \colon F(S_{\IA}(n)) \rightarrow \IA_n$ and
$\rho \colon F(S_{\Aut}(n)) \rightarrow \Aut(F_n)$ be the projections.  Then
for $s \in S_{\Aut}(n)$ and $w \in F(S_{\IA}(n))$, we have
\begin{equation}
\label{eqn:thetakey}
\pi(\ttheta(s)(w)) = \rho(s) \pi(w) \rho(s)^{-1} \in \IA_n.
\end{equation}
Our set of endomorphisms will then be
\[E_{\IA}(n) = \Set{$\ttheta(s)$}{$s \in S_{\Aut}(n)$}.\]
The relevance of the formula \eqref{eqn:thetakey} is that we want the induced
endomorphisms of our IA-presentation of $\IA_n$ to generate the image of
$\Aut(F_n)$ in $\Aut(\IA_n) \subset \End(\IA_n)$ arising from the conjugation action
of $\Aut(F_n)$ on $\IA_n$.

%\begin{remark}
%We use $\ttheta$ because we will later prove that $\ttheta$ induces a homomorphism
%$\theta$ from $\Aut(F_n)$ to the automorphism group of group with the L-presentation
%defined in this section; of course, we will eventually show that this group is
%$\IA_n$.
%\end{remark}

\paragraph{Defining $\ttheta$.}
Consider $s \in S_{\Aut}(n)$.  To define an endomorphism $\ttheta(s) \colon F(S_{\IA}(n)) \rightarrow F(S_{\IA}(n))$,
it is enough to say what $\ttheta(s)$ does to each element of $S_{\IA}(n)$.  There are two cases.
\begin{compactitem}
\item $s = P_{i,j}$ or $s = I_i$.  We then define $\ttheta(s)$ using
the action of $s$ on $F_n$ via formulas
\[\ttheta(s)(\Con{x_a}{x_b}) = \Con{s(x_a)}{s(x_b)} \quad \text{and} \quad \ttheta(s)(\Mulcomm{x_a^{\alpha}}{x_b^{\beta}}{x_c^{\gamma}}) = \Mulcomm{s(x_a^{\alpha})}{s(x_b^{\beta})}{s(x_c^{\gamma})}.\]
Here one should interpret $\Con{x_e^{-1}}{x_f}$ as $\Con{x_e}{x_f}$,
$\Con{x_e}{x_f^{-1}}$ as $\Con{x_e}{x_f}^{-1}$,
and $\Con{x_e^{-1}}{x_f^{-1}}$ as $\Con{x_e}{x_f}^{-1}$.
\item $s = \Mul{x_a^{\alpha}}{x_b}$.  
In this case, we define $\ttheta(s)$ via the formulas in Table \ref{table:thetadef}.  
These list the cases where $\ttheta(s)$ does not fix a generator, except that to avoid redundancy, we do not always list both a commutator transvection and its inverse.
Specifically, if $t=\Mulcomm{x_c^\gamma}{x_d^\delta}{x_e^\epsilon}$, possibly with $\{c,d,e\}\cap\{a,b\}\neq\varnothing$, and a formula is listed for $t'=\Mulcomm{x_c^\gamma}{x_e^\epsilon}{x_d^\delta}$ but not for $t$, then we define
\[\ttheta(s)(t)=\ttheta(s)(t')^{-1}.\]
If Table~\ref{table:thetadef} lists no entry for $t$ or $t'$, or the table lists no entry for $t$ and $t$ is a conjugation move, then we define $\ttheta(s)(t)=t$.
\end{compactitem}
These formulas were chosen to be as simple as possible, among formulas realizing equation~\eqref{eqn:thetakey}.
Just like for the relations, we recommend not dwelling on these formulas during one's first read through this paper.

\begin{proposition}\label{proposition:verifyendomorphisms}
The definition of $\ttheta$ satisfies equation~\eqref{eqn:thetakey}. 
\end{proposition}
This proof uses a computer verification and is postponed until \S \ref{section:somecomputations}.
Propositions~\ref{proposition:verifyrels} and~\ref{proposition:verifyendomorphisms} together imply that all of the extended relations from our L-presentation are trivial in $\IA_n$.
This means that the obvious map on generators (sending each generator to the automorphism it names) extends to a well defined homomorphism 
\[\LPres{S_{\IA}(n)}{R_{\IA}^0(n)}{E_{\IA}(n)} \to \IA_n.\]

%%%%%%%%%%%%%%%%%%%%%%%%%%%%%%%%%%%%%%%%%%%%%%%%%%%%%%%%%%%%%%%%%%%%%%%%%%%%%%%%%%%
%%%%%%%%%%%%%%%%%%%%%%%%%%%%%%%%%%%%%%%%%%%%%%%%%%%%%%%%%%%%%%%%%%%%%%%%%%%%%%%%%%%
%%%%%%%%%%%%%%%%%%%%%%%%%%%%%%%%%%%%%%%%%%%%%%%%%%%%%%%%%%%%%%%%%%%%%%%%%%%%%%%%%%%
\begin{table}[t!]
\begin{tabular}{l@{\hspace{28pt}}l}
\toprule
$s\in S_I$ & $\ttheta(\Mul{x_a^\alpha}{x_b}^{\beta})(s)$ \\
\midrule
$\Con{x_c}{x_a}$ & $(\Con{x_c}{x_a}^\alpha\Con{x_c}{x_b}^\beta)^\alpha$ \\
$\Con{x_a}{x_c}$ & $\Con{x_a}{x_c}\Mulcomm{x_a^\alpha}{x_b^{-\beta}}{x_c}$ \\
$\Con{x_b}{x_c}$ & $\Con{x_b}{x_c}\Mulcomm{x_a^\alpha}{x_b^{-\beta}}{x_c^{-1}}$\\
$\Con{x_b}{x_a}$ & $(\Con{x_a}{x_b}^\beta\Con{x_b}{x_a}^\alpha)^\alpha$ \\
$\Mulcomm{x_a^\alpha}{x_c^\gamma}{x_d^\delta}$ & $\Con{x_a}{x_b}^\beta\Mulcomm{x_a^\alpha}{x_c^\gamma}{x_d^\delta}\Con{x_a}{x_b}^{-\beta}$ \\
$\Mulcomm{x_c^\gamma}{x_a^\alpha}{x_d^\delta}$ & $\Mulcomm{x_c^\gamma}{x_b^\beta}{x_d^\delta}\Con{x_c}{x_b}^{-\beta}\Mulcomm{x_c^\gamma}{x_a^\alpha}{x_d^\delta}\Con{x_c}{x_b}^{\beta}$ \\
$\Mulcomm{x_c^\gamma}{x_a^{-\alpha}}{x_d^\delta}$ & $\Mulcomm{x_c^\gamma}{x_a^{-\alpha}}{x_d^\delta}\Con{x_c}{x_a}^\alpha\Mulcomm{x_c^\gamma}{x_b^{-\beta}}{x_d^\delta}\Con{x_c}{x_a}^{-\alpha}$ \\
$\Mulcomm{x_b^\beta}{x_c^\gamma}{x_d^\delta}$ & $\Con{x_a}{x_b}^\beta\Mulcomm{x_a^{-\alpha}}{x_c^\gamma}{x_d^\delta}\Mulcomm{x_b^\beta}{x_c^\gamma}{x_d^\delta}\Con{x_a}{x_b}^{-\beta}$ \\
$\Mulcomm{x_b^{-\beta}}{x_c^\gamma}{x_d^\delta}$ & $\Mulcomm{x_a^{\alpha}}{x_c^\gamma}{x_d^\delta}\Mulcomm{x_b^{-\beta}}{x_c^\gamma}{x_d^\delta}$ \\
$\Mulcomm{x_a^\alpha}{x_b^{\beta}}{x_c^\gamma}$ & $\Con{x_a}{x_b}^\beta\Mulcomm{x_a^\alpha}{x_b^{\beta}}{x_c^\gamma}\Con{x_a}{x_b}^{-\beta}$ \\
$\Mulcomm{x_a^\alpha}{x_b^{-\beta}}{x_c^\gamma}$ & $\Con{x_a}{x_b}^\beta\Mulcomm{x_a^\alpha}{x_b^{-\beta}}{x_c^\gamma}\Con{x_a}{x_b}^{-\beta}$ \\
$\Mulcomm{x_b^{\beta}}{x_a^\alpha}{x_c^\gamma}$ & $\Con{x_a}{x_c}^\gamma\Mulcomm{x_a^\alpha}{x_b^{-\beta}}{x_c^\gamma}\Con{x_b}{x_c}^{-\gamma}\Mulcomm{x_b^{-\beta}}{x_a^\alpha}{x_c^{-\gamma}}$ \\
$\Mulcomm{x_b^{\beta}}{x_a^{-\alpha}}{x_c^\gamma}$ & $\Con{x_c}{x_b}^{-\beta}\Con{x_c}{x_a}^{-\alpha}\Mulcomm{x_b^{-\beta}}{x_c^{-\gamma}}{x_a^\alpha}\Con{x_b}{x_c}^\gamma\Mulcomm{x_a^\alpha}{x_c^\gamma}{x_b^{-\beta}}\Con{x_a}{x_c}^{-\gamma}\Con{x_c}{x_a}^\alpha\Con{x_c}{x_b}^\beta$ \\
$\Mulcomm{x_b^{-\beta}}{x_a^{\alpha}}{x_c^\gamma}$ & $\Con{x_a}{x_c}^{-\gamma}\Con{x_c}{x_a}^\alpha\Mulcomm{x_b^{-\beta}}{x_c^\gamma}{x_a^{-\alpha}}\Con{x_c^\gamma}{x_b}^\beta\Mulcomm{x_a^\alpha}{x_c^\gamma}{x_b^{-\beta}}\Con{x_b}{x_c}^\gamma\Con{x_c}{x_b}^{-\beta}\Con{x_c}{x_a}^{-\alpha}$ \\
$\Mulcomm{x_b^{-\beta}}{x_a^{-\alpha}}{x_c^\gamma}$ & $\Con{x_b}{x_c}^{-\gamma}\Mulcomm{x_a^\alpha}{x_b^{-\beta}}{x_c^\gamma}\Con{x_c}{x_b}^{-\beta}\Mulcomm{x_b^{-\beta}}{x_a^{-\alpha}}{x_c^\gamma}\Con{x_c}{x_a}^{-\alpha}\Con{x_a}{x_c}^\gamma\Con{x_c}{x_a}^\alpha\Con{x_c}{x_b}^\beta$ \\
$\Mulcomm{x_c^\gamma}{x_a^{\alpha}}{x_b^{\beta}}$ & $\Con{x_a}{x_b}^\beta\Mulcomm{x_c^\gamma}{x_a^\alpha}{x_b^\beta}\Con{x_a}{x_b}^{-\beta}$\\
$\Mulcomm{x_c^\gamma}{x_a^{\alpha}}{x_b^{-\beta}}$ & $\Mulcomm{x_c^\gamma}{x_b^\beta}{x_a^\alpha}$\\
\bottomrule
\end{tabular}
\caption{Definition of $\ttheta(\Mul{x_a^{\alpha}}{x_b}^{\beta})$ on the generators $S_{\IA}(n)$.
All indices in each entry are assumed to be distinct.
If no entry is listed for $t\in S_{\IA}(n)$ or for the generator representing $t^{-1}$ (as in relation R0) then $\ttheta(\Mul{x_a^{\alpha}}{x_b}^{\beta})(t)=t$.
}
\label{table:thetadef}
\end{table}
%%%%%%%%%%%%%%%%%%%%%%%%%%%%%%%%%%%%%%%%%%%%%%%%%%%%%%%%%%%%%%%%%%%%%%%%%%%%%%%%%%%
%%%%%%%%%%%%%%%%%%%%%%%%%%%%%%%%%%%%%%%%%%%%%%%%%%%%%%%%%%%%%%%%%%%%%%%%%%%%%%%%%%%
%%%%%%%%%%%%%%%%%%%%%%%%%%%%%%%%%%%%%%%%%%%%%%%%%%%%%%%%%%%%%%%%%%%%%%%%%%%%%%%%%%%

\section{Tools for the proof}
\label{section:prooftools}

In this section, we assemble the tools we will need to prove Theorem \ref{maintheorem:finitelpres}.  
In \S \ref{section:presentation}, we discuss a theorem of the second author that gives a sort of
infinite presentation for a group acting on a simplicial complex.  
In \S \ref{section:partialbases}, we introduce the complex of partial bases.  
In \S \ref{section:simplexstabilizers}, we give generators for the $\IA_n$-stabilizers
of simplices in the complex of partial bases.
In \S \ref{section:autaction}, we introduce
an action of $\Aut(F_n)$ on the group given by our purported L-presentation for $\IA_n$.
Finally, in \S \ref{section:Lpreshom} we introduce a certain morphism between groups
given by L-presentations.

Two results in this sections have computer-aided proofs which are postponed until \S \ref{section:somecomputations}: 
Proposition \ref{proposition:autaction} from \S \ref{section:autaction} and Proposition \ref{proposition:mapfromkernel}
from \S \ref{section:Lpreshom}.

\subsection{Presentations from group actions}
\label{section:presentation}

Consider a group $G$ acting on a simplicial complex $X$.  We say that $G$ acts
{\em without rotations} if for all simplices $\sigma$ of $X$, 
the setwise and pointwise stabilizers of $\sigma$ coincide.  For a simplex
$\sigma$, denote by $G_{\sigma}$ the stabilizer of $\sigma$.  Letting
$X^{(0)}$ denote the vertex set of $X$, there is a homomorphism from the free product of vertex stabilizers
\[\psi \colon \BigFreeProd_{v \in X^{(0)}} G_v \longrightarrow G.\]
As notation, if $g \in G$ stabilizes a vertex $v$ of $X$, then
denote by $g_v$ the associated element of
\[G_v < \BigFreeProd_{v \in X^{(0)}} G_v.\]
The map $\psi$ is rarely injective.  Two families
of elements in its kernel are as follows.
\begin{compactitem}
\item If $e$ is an edge of $X$ joining vertices $v$ and $v'$ and
if $g \in G_e$, then $g_v g_{v'}^{-1} \in \Ker(\psi)$.  We call
these the {\em edge relators}.
\item If $v,w \in X^{(0)}$ and $g \in G_v$ and
$h \in G_w$, then $h_w g_v h_w^{-1} (h g h^{-1})_{h(v)}^{-1} \in \Ker(\psi)$.
We call these the {\em conjugation relators}.
\end{compactitem}
The second author gave hypotheses under which these generate $\Ker(\psi)$.

\begin{theorem}[{\cite{PutmanPresentation}}]
\label{theorem:presentation}
Consider a group $G$ acting without rotations on a $1$-connected simplicial complex $X$.  Assume
that $X/G$ is $2$-connected.  Then the kernel of the map $\psi$ described above
is normally generated by the edge and conjugation relators.
\end{theorem}

\subsection{The complex of partial bases}
\label{section:partialbases}

We now introduce the simplicial complex to which we will apply Theorem \ref{theorem:presentation}.  
For $z \in F_n$, let $\Conj{z}$ denote the union of the conjugacy classes of $z$ and $z^{-1}$.  

\begin{definition}
A {\em partial basis} for $F_n$ is a set $\{z_1,\ldots,z_k\} \subset F_n$ such that there
exist $z_{k+1},\ldots,z_n \in F_n$ with $\{z_1,\ldots,z_n\}$ a free basis for $F_n$.
The {\em complex of partial bases for $F_n$}, denoted $\Bases_n$, is the simplicial complex whose $(k-1)$-simplices
are sets $\{\Conj{z_1},\ldots,\Conj{z_k}\}$, where $\{z_1,\ldots,z_k\}$ is a partial basis
for $F_n$.
\end{definition}

The group $\Aut(F_n)$ acts on $\Bases_n$, and we wish to apply Theorem \ref{theorem:presentation}
to the restriction of this action to $\IA_n$.  It is clear that $\IA_n$ acts on $\Bases_n$ without
rotations, so we must check that $\Bases_n$ is $1$-connected and that $\Bases_n / \IA_n$ is $2$-connected.

We start by verifying that $\Bases_n$ is $1$-connected.

\begin{proposition}
\label{proposition:basescon}
The simplicial complex $\Bases_n$ is $1$-connected for $n \geq 3$.  
\end{proposition}
\begin{proof}
For $z \in F_n$, let $\Conj{z}'$ be the conjugacy class of $z$.  Define $\Bases_n'$ to be the simplicial
complex whose $(k-1)$-simplices
are sets $\{\Conj{z_1}',\ldots,\Conj{z_k}'\}$, where $\{z_1,\ldots,z_k\}$ is a partial basis for $F_n$.
In \cite{DayPutmanComplex}, the authors proved that $\Bases_n'$ is $1$-connected for $n \geq 3$.  There is
a natural simplicial map $\rho\colon \Bases_n' \rightarrow \Bases_n$.  Letting 
$\psi^0\colon (\Bases_n)^{(0)} \rightarrow (\Bases_n')^{(0)}$ be an arbitrary map satisfying
$\rho \circ \psi^0 = \text{id}$, it is clear that $\psi^0$ extends to a simplicial map
$\psi\colon \Bases_n \rightarrow \Bases_n'$ satisfying $\rho \circ \psi = \text{id}$.  This implies
that $\rho$ induces a surjection on all homotopy groups, so $\Bases_n$ is 
$1$-connected for $n \geq 3$.
\end{proof}

It also follows from \cite{DayPutmanComplex} that $\Bases_n / \IA_n$ is $(n-2)$-connected.  In particular,
it is $2$-connected for $n \geq 4$, and thus satisfies the conditions of Theorem \ref{theorem:presentation}
for $n \geq 4$.  However, we will need a complex that satisfies the conditions of
Theorem \ref{theorem:presentation} for $n=3$ as well.  We therefore attach cells to increase the connectivity.

\begin{definition}
The {\em augmented complex of partial bases for $F_n$}, denoted $\ABases_n$, is the simplicial
complex whose $(k-1)$-simplices are as follows.
\begin{compactitem}
\item Sets of the form $\{\Conj{z_1},\ldots,\Conj{z_k}\}$, where $\{z_1,\ldots,z_k\}$ is a partial basis
for $F_n$.  These will be called the {\em standard simplices}.
\item Sets of the form $\{\Conj{z_1 z_2}, \Conj{z_1},\Conj{z_2},\ldots,\Conj{z_{k-1}}\}$,
where $\{z_1,\ldots,z_{k-1}\}$ is a partial basis for $F_n$.
These will be called the {\em additive simplices}.
\end{compactitem}
\end{definition}

\begin{remark}
Since $z_1 z_2$ and $z_2 z_1$ are conjugate, the two additive simplices
\[\{\Conj{z_1 z_2}, \Conj{z_1},\Conj{z_2},\ldots,\Conj{z_{k-1}}\} \quad \text{and} \quad
\{\Conj{z_2 z_1}, \Conj{z_1},\Conj{z_2},\ldots,\Conj{z_{k-1}}\}\] 
of $\ABases_n$ are the same.
\end{remark}

The group $\Aut(F_n)$ (and hence $\IA_n$) still acts on $\ABases_n$.  
Since $\ABases_n$ is obtained from $\Bases_n$ by adding simplices of dimension at least $2$, it
inherits the $1$-connectivity of $\Bases_n$ for $n \geq 3$ asserted in Proposition \ref{proposition:basescon}.

\begin{proposition}
\label{proposition:abasescon}
The complex $\ABases_n$ is $1$-connected for $n \geq 3$.
\end{proposition}

To help us understand the connectivity of $\ABases_n / \IA_n$, we introduce the following complex.
For $\vec{v} \in \Z^n$, let $\Lax{\vec{v}}$ denote the set $\{\vec{v},-\vec{v}\}$.

\begin{definition}
A {\em partial basis} for $\Z^n$ is a set $\{\vec{v}_1,\ldots,\vec{v}_k\} \subset \Z^n$ such that there
exist $\vec{v}_{k+1},\ldots,\vec{v}_n \in \Z^n$ with $\{\vec{v}_1,\ldots,\vec{v}_n\}$ a basis for $\Z^n$.
The {\em augmented complex of lax partial bases for $\Z^n$}, denoted $\ABases_n(\Z)$, is the simplicial
complex whose $(k-1)$-simplices are as follows.
\begin{compactitem}
\item Sets of the form $\{\Lax{\vec{v}_1},\ldots,\Lax{\vec{v}_k}\}$, where $\{\vec{v}_1,\ldots,\vec{v}_k\}$ is a partial basis
for $\Z^n$.  These will be called the {\em standard simplices}.
\item Sets of the form $\{\Lax{\vec{v}_1 + \vec{v}_2},\Lax{\vec{v}_1},\Lax{\vec{v}_2},\ldots,\Lax{\vec{v}_{k-1}}\}$,
where $\{\vec{v}_1,\ldots,\vec{v}_{k-1}\}$ is a partial basis for $\Z^n$.
These will be called the {\em additive simplices}.
\end{compactitem}
\end{definition}

\noindent
We then have the following lemma.

\begin{lemma}
\label{lemma:identifyquotient}
We have $\ABases_n / \IA_n \cong \ABases_n(\Z)$ for $n \geq 1$.
\end{lemma}

\noindent
For the proof of Lemma \ref{lemma:identifyquotient}, we will need the following result of the authors.
For $z \in F_n$, let $[z] \in \Z^n$ be the associated element of the abelianization of $F_n$.

\begin{lemma}[{\cite[Lemma 5.3]{DayPutmanComplex}}]
\label{lemma:basiscompletion}
Let $\{\vec{v}_1,\ldots,\vec{v}_n\}$ be a basis for $\Z^n$ and let $\{z_1,\ldots,z_k\}$ be a partial
basis for $F_n$ such that $[z_i] = \vec{v}_i$ for $1 \leq i \leq k$.  Then there exists $z_{k+1},\ldots,z_n \in F_n$
with $[z_i] = \vec{v}_i$ for $k+1 \leq i \leq n$ such that $\{z_1,\ldots,z_n\}$ is a basis for $F_n$.
\end{lemma}

\begin{proof}[{Proof of Lemma \ref{lemma:identifyquotient}}]
The map $(\ABases_n)^{(0)} \rightarrow (\ABases_n(\Z))^{(0)}$ that takes $\Conj{z}$ to $[z]$ extends to
a simplicial map $\rho\colon \ABases_n \rightarrow \ABases_n(\Z)$.  Since $\IA_n$ acts without rotations
on $\ABases_n$, the quotient $\ABases_n/\IA_n$ has a natural CW-complex structure whose $k$-cells
are the $\IA_n$-orbits of the $k$-cells of $\ABases_n$ (warning: though it will turn out that in this
case it is, this CW-complex structure need not be a simplicial complex structure; consider, for example,
the action of $\Z$ by translations on the standard triangulation of $\R$ whose vertices are $\Z$).  
Since $\rho$ is $\IA_n$-invariant,
it factors through a map $\overline{\rho}\colon \ABases_n/\IA_n \rightarrow \ABases_n(\Z)$.  We will
prove that $\overline{\rho}$ is an isomorphism of CW-complexes.

This requires checking two things.  The first is that every simplex of $\ABases_n(\Z)$ is in the image
of $\rho$, which is an immediate consequence of Lemma \ref{lemma:basiscompletion}.  The second
is that if $\sigma$ and $\sigma'$ are simplices of $\ABases_n$ such that $\rho(\sigma) = \rho(\sigma')$, then
there exists some $f \in \IA_n$ such that $f(\sigma) = \sigma'$.  It is clear that $\sigma$ and $\sigma'$
are either both standard simplices or both additive simplices.  Assume first that they are both
standard simplices.  We can then write
\[\sigma = \{\Conj{z_1},\ldots,\Conj{z_k}\} \quad \text{and} \quad \sigma' = \{\Conj{z_1'},\ldots,\Conj{z_k'}\}\]
as in the definition of standard simplices  with $[z_i] = [z_i']$ for $1 \leq i \leq k$.  
Set $\vec{v}_i = [z_i] = [z_i']$ for $1 \leq i \leq k$.  The
set $\{\vec{v}_1,\ldots,\vec{v}_k\}$ is a partial basis for $\Z^n$, so we can extend it to a basis
$\{\vec{v}_1,\ldots,\vec{v}_n\}$.  Applying Lemma \ref{lemma:basiscompletion} twice, we can find
$z_{k+1},\ldots,z_n \in F_n$ and $z_{k+1}',\ldots,z_n' \in F_n$ such that $[z_i] = [z_i'] = \vec{v}_i$ for
$k+1 \leq i \leq n$ and such that both $\{z_1,\ldots,z_n\}$ and $\{z_1',\ldots,z_n'\}$ are free bases
for $F_n$.  There then exists $f \in \Aut(F_n)$ such that $f(z_i) = z_i'$ for $1 \leq i \leq n$.  By construction, we have
$f \in \IA_n$ and $f(\sigma) = \sigma'$.

It remains to deal with the case where $\sigma$ and $\sigma'$ are both simplices of additive type.  Write
\[\sigma = \{\Conj{z_1 z_2}, \Conj{z_1},\Conj{z_2},\ldots,\Conj{z_{k-1}}\} \quad \text{and} \quad \sigma' = \{\Conj{z_1' z_2'},\Conj{z_1'},\ldots,\Conj{z_{k-1}'}\}\]
as in the definition of additive simplices.  The unordered sets $\{\Lax{[z_1]+[z_2]},\Lax{[z_1]},\Lax{[z_2]}\}$ and
$\{\Lax{[z_1']+[z_2']},\Lax{[z_1']},\Lax{[z_2']}\}$ are minimal nonempty subsets of $\rho(\sigma) = \rho(\sigma')$ such
that the defining elements of $\Z^n$ are not linearly independent.  It follows that as unordered sets
we have
\[\rho(\{\Conj{z_1 z_2}, \Conj{z_1},\Conj{z_2}\}) = \rho(\{\Conj{z_1' z_2'}, \Conj{z_1'},\Conj{z_2'}\})\]
and
\[\rho(\{\Conj{z_3},\ldots,\Conj{z_{k-1}}\}) = \rho(\{\Conj{z_3'},\ldots,\Conj{z_{k-1}'}\}).\]
Reordering the $z_i$ and possibly replacing some of the $z_i$ by $z_i^{-1}$ (which does not change
$\Conj{z_i}$), we can assume that $[z_i] = [z_i']$ for $3 \leq i \leq k-1$.  

The next observation is that all of the following sets define the same additive simplex (but with the vertices
in a different order; all six possible orderings occur):
\begin{align*}
&\{\Conj{z_1 z_2},\Conj{z_1},\Conj{z_2}\}, \{\Conj{z_2 z_1},\Conj{z_2},\Conj{z_1}\},
\{\Conj{z_1},\Conj{z_1 z_2},\Conj{z_2^{-1}}\},\\
&\quad\quad\quad\{\Conj{z_1^{-1}},\Conj{z_2},\Conj{z_2^{-1} z_1^{-1}}\},
\{\Conj{z_2},\Conj{z_2 z_1},\Conj{z_1^{-1}}\}, \{\Conj{z_2^{-1}},\Conj{z_1},\Conj{z_1^{-1} z_2^{-1}}\}.
\end{align*}
By reordering $\sigma$ and possibly changing some of our expressions for the elements in it again, we can
assume that
\[\Lax{[z_1]+[z_2]} = \Lax{[z_1']+[z_2']} \quad \text{and} \quad \Lax{[z_1]}=\Lax{[z_1']} \quad \text{and} \quad
\Lax{[z_2]}=\Lax{[z_2']}\]
and that $[z_i] = [z_i']$ for $3 \leq i \leq k-1$.

The final observation is that either $([z_1],[z_2]) = ([z_1'],[z_2'])$ or
$([z_1],[z_2])=(-[z_1'],-[z_2'])$; the
key point here is that changing the sign of one of $\{[z_1],[z_2]\}$ but not the other changes
$\Lax{[z_1]+[z_2]}$.  If the second possibility occurs, then replace $z_1$ and $z_2$ with $z_1^{-1}$ and $z_2^{-1}$,
respectively; this does not change $\sigma$.  The upshot is that we now have arranged for
$[z_i] = [z_i']$ for all $1 \leq i \leq k-1$.  By the same argument we used to deal with standard
simplices, there exists some $f \in \IA_n$ such that $f(z_i) = z_i'$ for $1 \leq i \leq k-1$.  Since
$f(z_1 z_2) = z_1' z_2'$, we see that $f(\sigma) = \sigma'$, as desired.
\end{proof}

The second author together with Church proved in \cite{ChurchPutmanCodim1} that
$\ABases_n(\Z)$ is $(n-1)$-connected for $n \geq 1$.  We therefore deduce the following.

\begin{proposition}
\label{proposition:abasesiacon}
The complex $\ABases_n / \IA_n$ is $(n-1)$-connected for $n \geq 1$.
\end{proposition}

\subsection{Generators for simplex stabilizers}
\label{section:simplexstabilizers}

This section is devoted to the following proposition, which gives generators
for the stabilizers in $\IA_n$ of simplices of $\Bases_n$.  Recall that
$S_{\Mag}(n)$ is Magnus's generating set for $\IA_n$ discussed in the introduction.

\begin{proposition}
\label{proposition:stabilizergenerators}
Fix $1 \leq k \leq n$ and define $\Gamma = (\IA_n)_{\Conj{x_{n-k+1}},\Conj{x_{n-k+2}},\ldots,\Conj{x_n}}$.  Then $\Gamma$
is generated by
\begin{align*}
S_{\Mag}(n) \cap \Gamma &= \Set{$\Con{x_a}{x_b}$}{$1\leq a,b\leq n$ distinct}\\
&\quad\quad\cup\Set{$\Mulcomm{x_a}{x_b}{x_c}$}{$1\leq a \leq n-k$, $1\leq b,c\leq n$ distinct}.
\end{align*}
\end{proposition}
\begin{proof}
The map $F_n \rightarrow F_{n-k}$
that quotients out by the normal closure of $\{x_{n-k+1},\ldots,x_n\}$ induces a split surjection
$\rho \colon \Gamma \rightarrow \IA_{n-k}$.  Define $\BirKer_{n-k,k} = \ker(\rho)$, so we have 
$\Gamma = \BirKer_{n-k,k} \rtimes \IA_{n-k}$.  As we said in the introduction, Magnus \cite{MagnusGenerators}
proved that $\IA_{n-k}$ is generated by 
\begin{equation}
\label{eqn:partialgenset1}
\Set{$\Con{x_a}{x_b}$}{$1\leq a,b\leq n-k$ distinct} \cup\Set{$\Mulcomm{x_a}{x_b}{x_c}$}{$1\leq a,b,c \leq n-k$ distinct}.
\end{equation}
The authors proved in \cite[Theorem~A]{DayPutmanBirmanIA} that $\BirKer_{n-k,k}$ is generated by
\begin{equation}
\label{eqn:partialgenset2}
\begin{split}
&\Set{$\Con{x_a}{x_b}$}{$n-k+1\leq a \leq n$, $1 \leq b \leq n$ distinct}\\
&\quad\quad \cup \Set{$\Con{x_a}{x_b}$}{$1 \leq a \leq n$, $n-k+1 \leq b \leq n$ distinct}\\
&\quad\quad \cup \Set{$\Mulcomm{x_a}{x_b}{x_c}$}{$1 \leq a \leq n-k$, $n-k-1 \leq b \leq n$, $1 \leq c \leq n$ distinct}.
\end{split}
\end{equation}

The union of \eqref{eqn:partialgenset1} and \eqref{eqn:partialgenset2} is the claimed generating set for $\Gamma$.
\end{proof}

\begin{remark}
\label{remark:conjugacyclass}
For $z \in F_n$, define $\Conj{z}'$ to be the conjugacy class of $z$.  The reference \cite{DayPutmanBirmanIA} actually
deals with $(\IA_n)_{\Conj{x_{n-k+1}}',\Conj{x_{n-k+2}}',\ldots,\Conj{x_n}'}$ instead of
$(\IA_n)_{\Conj{x_{n-k+1}},\Conj{x_{n-k+2}},\ldots,\Conj{x_n}}$; however, since $x_i$ and $x_i^{-1}$ have different
images in $F_n^{\text{ab}}$ these two stabilizer subgroups are actually equal.
There are also notational differences: the group denoted $\BirKer_{n-k,k}$ here is denoted $\BirKer_{n-k,k}^{\IA}$ in that paper.
\end{remark}

\subsection{The action of \texorpdfstring{$\Aut(F_n)$}{Aut(Fn)}}
\label{section:autaction}

Let $\Quotient_n$ be the group with the L-presentation $\LPres{S_{\IA}(n)}{R_{\IA}^0(n)}{E_{\IA}(n)}$
discussed in \S \ref{section:thepresentation}.  
By Propositions~\ref{proposition:verifyrels} and~\ref{proposition:verifyendomorphisms}, there is a map map $\pi \colon \Quotient_n \rightarrow \IA_n$.
The group $\Aut(F_n)$ acts on $\IA_n$ by conjugation.  The goal of this section is
to state Proposition \ref{proposition:autaction} below, which asserts that this action can be lifted to $\Quotient_n$.

To state some important properties of this lifted action, we must introduce some notation.
First, let $S_{\Aut}(n) \subset \Aut(F_n)$ be the generating
set discussed in \S \ref{section:thepresentation}.  Recall that $E_{\IA}(n) \subset \End(F(S_{\IA}(n)))$
is the image of a map $\ttheta \colon S_{\Aut}(n) \rightarrow \End(F(S_{\IA}(n)))$.  
There is thus a map
$e \colon S_{\Aut}(n) \rightarrow \End(\Quotient_n)$ whose image is the set of induced endomorphisms of
our L-presentation.  It will turn out that the image of $e$ consists of automorphisms and these
automorphisms generate the action of $\Aut(F_n)$ on $\Quotient_n$.

Second, recall that $\{x_1,\ldots,x_n\}$
is a fixed free basis for $F_n$.  Let 
\begin{align*}
(S_{\IA}(n))_{\Conj{x_n}} = &\Set{$\Con{x_a}{x_b}$}{$1 \leq a,b \leq n$ distinct} \\
& \cup \Set{$\Mulcomm{x_a^{\alpha}}{x_b^{\beta}}{x_c^{\gamma}}$}{$1 \leq a,b,c \leq n$ distinct, $\alpha,\beta,\gamma \in \{\pm 1\}$, $a \neq n$}.
\end{align*}
This is exactly the subset of $S_{\IA}(n) \subset \IA_n$ consisting of automorphisms that fix $\Conj{x_n}$;
Proposition~\ref{proposition:stabilizergenerators} (with $k=1$) implies that it generates the stabilizer subgroup $(\IA_n)_{\Conj{x_n}}$.
Define $(\Quotient_n)_{\Conj{x_n}}$ be the subgroup of $\Quotient_n$ generated by $(S_{\IA}(n))_{\Conj{x_n}}$.  We will
then require the stabilizer subgroup $(\Aut(F_n))_{\Conj{x}}$ to preserve the subgroup
$(\Quotient_n)_{\Conj{x_n}}$.

Our proposition is as follows.

\begin{proposition}
\label{proposition:autaction}
For all $n \geq 2$, there is an action of $\Aut(F_n)$ on $\Quotient_n$ that satisfies the following
three properties.
\vspace*{-2ex}\begin{enumerate}\setlength{\itemsep}{0ex}
\item The action comes from the induced endomorphisms in the sense
that for $s \in S_{\IA}(n) \subset \Aut(F_n)$ and $q \in \Quotient_n$, we have $s \cdot q = e(s) \cdot q$.
\item The restriction of the action to $\IA_n$ induces the conjugation action of $\Quotient_n$
on itself in the sense that for $q,r \in \Quotient_n$, we have $\pi(r) \cdot q = r q r^{-1}$.
\item For $\eta \in (\Aut(F_n))_{\Conj{x_n}}$ and $q \in (\Quotient_n)_{\Conj{x_n}}$, we have
$\eta \cdot q \in (\Quotient_n)_{\Conj{x_n}}$.
\end{enumerate}
\end{proposition}

The proof of Proposition \ref{proposition:autaction} is a computation with generators and relations (mostly done
by computer), so we have postponed it until \S \ref{section:somecomputations}.  

\subsection{A homomorphism between L-presentations}
\label{section:Lpreshom}

There is a natural split surjection $\rho \colon (\IA_n)_{\Conj{x_n}} \rightarrow \IA_{n-1}$ arising from
the map $F_n \rightarrow F_{n-1}$ which quotients out the normal closure of $x_n$.  Let
$\BirKer_{n-1,1} = \Ker(\rho)$, so we have a decomposition $(\IA_n)_{\Conj{x_n}} = \BirKer_{n-1,1} \rtimes \IA_{n-1}$.
Building on the Birman exact sequence for $\Aut(F_n)$ we constructed in \cite{DayPutmanBirman}, we constructed an L-presentation for 
$\BirKer_{n-1,1}$ in~\cite[Theorem~D]{DayPutmanBirmanIA} ($\BirKer_{n-1,1}$ is denoted $\BirKer_{n-1,1}^{\IA}$ in that paper).
This L-presentation plays a crucial role in the inductive step of our proof, because it allows us to obtain the following proposition:
\begin{proposition}\label{proposition:mapfromkernel}
There is a homomorphism $\BirKer_{n-1,1}\to \LPres{S_{\IA}(n)}{R_{\IA}^0(n)}{E_{\IA}(n)}$ fitting into the following commuting triangle:
\[
\xymatrix{
\BirKer_{n-1,1} \ar[r] \ar@{_{(}->}[dr] & \LPres{S_{\IA}(n)}{R_{\IA}^0(n)}{E_{\IA}(n)} \ar@{->>}[d] \\
& \IA_n
}
\]
\end{proposition}

Usually finding a homomorphism between groups given by presentations is simple: one checks that the relations map to products of conjugates of relations.
This is the spirit of the proof of Proposition~\ref{proposition:mapfromkernel}, but the substitution rules and extended relations complicate the picture.
Our proof of Proposition \ref{proposition:mapfromkernel} is computer-assisted and is postponed until \S \ref{section:somecomputations}.

\section{Verification of our L-presentation}
\label{section:theproof}

In this section, we prove Theorem \ref{maintheorem:finitelpres}, which says that
$\IA_n$ has the finite L-presentation $\LPres{S_{\IA}(n)}{R_{\IA}^0(n)}{E_{\IA}(n)}$
discussed in \S \ref{section:thepresentation}.  Our
proof is inspired by the proof of the main theorem of \cite{BrendleMargalitPutman}.  We will make use of
Propositions \ref{proposition:verifyrels}, \ref{proposition:verifyendomorphisms},
\ref{proposition:autaction}, and~\ref{proposition:mapfromkernel}, which were all stated in previous sections
and which will be proved (with the aid of a computer) in \S \ref{section:somecomputations}.

\begin{proof}[{Proof of Theorem \ref{maintheorem:finitelpres}}]
As notation, let $\Quotient_n$ be the group given by $\LPres{S_{\IA}(n)}{R_{\IA}^0(n)}{E_{\IA}(n)}$.  Elements
of $S_{\IA}(n)$ play dual roles as elements of $\Quotient_n$ and as elements of $\IA_n$, and during our
proof it will be important to distinguish them.  Therefore, throughout this proof elements
$\Con{x_a}{x_b}$ and $\Mulcomm{x_a^{\alpha}}{x_b^{\beta}}{x_c^{\gamma}}$ will always lie in $\IA_n$; the associated elements
of $\Quotient_n$ will be denoted $\QCon{x_a}{x_b}$ and $\QMulcomm{x_a^{\alpha}}{x_b^{\beta}}{x_c^{\gamma}}$.

There is a natural projection map $\pi\colon \Quotient_n \rightarrow \IA_n$.  We will prove that $\pi$
is an isomorphism by induction on $n$.
The base cases are $n=1$ and $n=2$.  For $n=1$, both
$\IA_n$ and $\Quotient_n$ are the trivial group, so there is nothing to prove.  For $n=2$, it is a classical
theorem of Nielsen (\cite{NielsenIA2}; see also \cite[Proposition 4.5]{LyndonSchupp}) that $\IA_2$
is the group of inner automorphisms of $F_2$, so $\IA_2$ is a free group on the
generators $\Con{x_1}{x_2}$ and $\Con{x_2}{x_1}$.  
Our generating set for $\Quotient_2$ is $\{\QCon{x_1}{x_2}, \QCon{x_2}{x_1}\}$, and for $n=2$, the set of basic relations $R^0_{\IA}(2)$ is empty.
Even though our set of substitution rules $E_{\IA}(2)$ is nonempty, it follows that our full set of relations for $\Quotient_2$ is empty.
So our presentation for
$\Quotient_2$ is $\Pres{\QCon{x_1}{x_2}, \QCon{x_2}{x_1}}{\emptyset}$, and the result
is also true in this case.

Assume now that $n \geq 3$ and that the projection map $\Quotient_{n'} \rightarrow \IA_{n'}$ is
an isomorphism for all $1 \leq n' < n$.  Since $\pi$ is a surjection,
to prove that $\pi$ is an isomorphism it is enough to construct a homomorphism $\phi \colon \IA_n \rightarrow \Quotient_n$
such that $\phi \circ \pi = \text{id}$.  Propositions \ref{proposition:abasescon} and \ref{proposition:abasesiacon} 
show that the action of $\IA_n$ on $\ABases_n$ satisfies the conditions of Theorem \ref{theorem:presentation}, so
\[\IA_n \cong \left(\BigFreeProd_{\Conj{z} \in (\ABases_n)^{(0)}} (\IA_n)_{\Conj{z}}\right)/R,\]
where $R$ is the normal closure of the edge and conjugation relators.  The construction of $\phi$
will have two steps.  First, we will use 
the action of $\Aut(F_n)$ on $\Quotient_n$ provided by Proposition \ref{proposition:autaction} to construct
a map
\[\widetilde{\phi} \colon \BigFreeProd_{\Conj{z} \in (\ABases_n)^{(0)}} (\IA_n)_{\Conj{z}} \longrightarrow \Quotient_n.\]
Second, we will show that $\widetilde{\phi}$ takes the edge and conjugation relators to $1$, and thus
induces a map $\phi \colon \IA_n \rightarrow \Quotient_n$.  We will close by verifying that $\phi \circ \pi = \text{id}$.

\paragraph{Construction of $\widetilde{\phi}$.}
To construct $\widetilde{\phi}$, we must construct a map
\[\widetilde{\phi}_{\Conj{z}} \colon (\IA_n)_{\Conj{z}} \longrightarrow \Quotient_n\]
for each vertex $\Conj{z}$ of $\ABases_n$.  Recalling that $\{x_1,\ldots,x_n\}$ is our fixed
free basis for $F_n$, we begin with the vertex $\Conj{x_n}$.  In the following claim,
we will use the notation $(S_{\IA}(n))_{\Conj{x_n}}$ and $(\Quotient_n)_{\Conj{x_n}}$ introduced
in \S \ref{section:autaction}.

\BeginClaims
\begin{claims}
\label{claim:phixn}
The restriction of $\pi$ to $(\Quotient_n)_{\Conj{x_n}}$ is an isomorphism onto $(\IA_n)_{\Conj{x_n}}$.
\end{claims}
\begin{proof}[Proof of claim]
Proposition~\ref{proposition:stabilizergenerators} implies that natural map
$\pi|_{(\Quotient_n)_{\Conj{x_n}}} \co (\Quotient_n)_{\Conj{x_n}} \rightarrow (\IA_n)_{\Conj{x_n}}$ is surjective, since the generators from that proposition (with $k=1$) are in the image.

Our inductive hypothesis says that the map $\pi|_{\Quotient_{n-1}}\colon \Quotient_{n-1} \rightarrow \IA_{n-1}$
is an isomorphism.  Recall from \S \ref{section:Lpreshom} that
$(\IA_n)_{\Conj{x_n}} = \BirKer_{n-1,1} \rtimes \IA_{n-1}$, where the projection
$(\IA_n)_{\Conj{x_n}} \rightarrow \IA_{n-1}$ is the one induced by the map $F_n \rightarrow F_{n-1}$ that quotients
out by the normal closure of $x_n$, and $\BirKer_{n-1,1}$ is the kernel of this projection.  The composition
\[(\Quotient_n)_{\Conj{x_n}} \longrightarrow (\IA_n)_{\Conj{x_n}} \longrightarrow \IA_{n-1} 
\overset{\pi|_{\Quotient_{n-1}}^{-1}}{\underset{\cong}\longrightarrow} \Quotient_{n-1}\]
is a well defined homomorphism.
It is a composition of surjective maps, and is therefore surjective.
We define $\mathfrak{K}_{n-1,1}$ to be the kernel of this composition of maps.

The restriction of $\pi$ to $\mathfrak{K}_{n-1,1}$ has its image in $\BirKer_{n-1,1}$ 
since the map $(\Quotient_n)_{\Conj{x_n}}\to \Quotient_{n-1}$ factors through $(\IA_n)_{\Conj{x_n}}\to \IA_{n-1}$.
Proposition \ref{proposition:stabilizergenerators} says that $\BirKer_{n-1,1}$ is generated by the set
\begin{align*}
S_{\BirKer}(n) := &\Set{$\Con{x_n}{x_a}, \Con{x_a}{x_n}$}{$1 \leq a < b$} \\
& \cup \Set{$\Mulcomm{x_a^{\alpha}}{x_b^{\beta}}{x_n^{\gamma}}, \Mulcomm{x_a^{\alpha}}{x_n^{\gamma}}{x_b^{\beta}}$}{$1 \leq a,b < n$ distinct, $\alpha,\beta,\gamma \in \{\pm 1\}$}.
\end{align*}
Since these generators are contained in $\mathfrak{K}_{n-1,1}$, the map
$\mathfrak{K}_{n-1,1}\to\BirKer_{n-1,1}$ is surjective.
Further, Proposition~\ref{proposition:mapfromkernel} gives us a left inverse to $\pi|_{\mathfrak{K}_{n-1,1}}$.
We conclude that $\pi|_{\mathfrak{K}_{n-1,1}}$ is an isomorphism $\mathfrak{K}_{n-1,1}\cong  \BirKer_{n-1,1}$.
We note that existence of this isomorphism is a deceptively difficult part of the proof, and it is the main consequence that we draw from~\cite{DayPutmanBirmanIA}.

  Summing up,
we have a commutative diagram of short exact sequences as follows.
\[\begin{CD}
1 @>>> \mathfrak{K}_{n-1,1} @>>> (\Quotient_n)_{\Conj{x_n}} @>>> \Quotient_{n-1} @>>> 1 \\
@.     @VV{\cong}V         @VVV                            @VV{\cong}V          @.\\
1 @>>> \BirKer_{n-1,1}  @>>> (\IA_n)_{\Conj{x_n}}       @>>> \IA_{n-1}       @>>> 1
\end{CD}\]
The five-lemma therefore says that the projection map $(\Quotient_n)_{\Conj{x_n}} \rightarrow (\IA_n)_{\Conj{x_n}}$
is an isomorphism, as desired.
\end{proof}

Claim \ref{claim:phixn} implies that we can define a map 
$\widetilde{\phi}_{\Conj{x_n}} \colon (\IA_n)_{\Conj{x_n}} \rightarrow \Quotient_n$ via the formula
$\widetilde{\phi}_{\Conj{x_n}} = (\pi|_{(\Quotient_n)_{\Conj{x_n}}})^{-1}$.

Now consider a general vertex $\Conj{z}$ of $\ABases_n$.  Here we will use the action
of $\Aut(F_n)$ on $\Quotient_n$ provided by Proposition \ref{proposition:autaction}.
The group $\Aut(F_n)$ acts transitively
on the set of primitive elements of $F_n$, so there exists some $\nu \in \Aut(F_n)$ such that
$\nu(x_n) = z$.  We then define a map $\widetilde{\phi}_{\Conj{z}} \colon (\IA_n)_{\Conj{z}} \rightarrow \Quotient_n$
via the formula
\[\widetilde{\phi}_{\Conj{z}}(\eta) = \nu \cdot \widetilde{\phi}_{\Conj{x_n}}(\nu^{-1} \eta \nu) \quad \quad (\eta \in (\IA_n)_{\Conj{z}}).\]
This appears to depend on the choice of $\nu$, but the following claim says that this choice does not
matter.

\begin{claims}
\label{claim:nuindep}
The map $\widetilde{\phi}_{\Conj{z}}(\eta)$ does not depend on the choice of $\nu$.
\end{claims}
\begin{proof}[Proof of claim]
Assume that $\nu_1,\nu_2 \in \Aut(F_n)$ both satisfy $\nu_i(x_n) = z$, and consider some
$\eta \in (\IA_n)_{\Conj{z}}$.  Our goal is to prove that
\begin{equation}
\label{eqn:nuindep1}
\nu_1 \cdot \widetilde{\phi}_{\Conj{x_n}}(\nu_1^{-1} \eta \nu_1) = \nu_2 \cdot \widetilde{\phi}_{\Conj{x_n}}(\nu_2^{-1} \eta \nu_2).
\end{equation}
Define $\mu = \nu_1^{-1} \nu_2$ and $\omega = \nu_2^{-1} \eta \nu_2$, so $\mu \in (\Aut(F_n))_{\Conj{x_n}}$ and
$\omega \in (\IA_n)_{\Conj{x_n}}$.  We will first prove that
\begin{equation}
\label{eqn:nuindep2}
\widetilde{\phi}_{\Conj{x_n}}(\mu \omega \mu^{-1}) = \mu \cdot \widetilde{\phi}_{\Conj{x_n}}(\omega).
\end{equation}
To see this, observe first that by construction both $\widetilde{\phi}_{\Conj{x_n}}(\mu \omega \mu^{-1})$
and $\widetilde{\phi}_{\Conj{x_n}}(\omega)$ lie in $(\Quotient_n)_{\Conj{x_n}}$.  The third part
of Proposition \ref{proposition:autaction} implies that $\mu \cdot \widetilde{\phi}_{\Conj{x_n}}(\omega)$
also lies in $(\Quotient_n)_{\Conj{x_n}}$.  Claim \ref{claim:phixn} says that
$\pi|_{(\Quotient_n)_{\Conj{x_n}}}$ is injective, so to prove \eqref{eqn:nuindep2}, it is thus 
enough to prove that $\widetilde{\phi}_{\Conj{x_n}}(\mu \omega \mu^{-1})$ and 
$\mu \cdot \widetilde{\phi}_{\Conj{x_n}}(\omega)$ have the same image under $\pi$.  This follows from
the calculation
\[\pi(\widetilde{\phi}_{\Conj{x_n}}(\mu \omega \mu^{-1})) = \mu \omega \mu^{-1} = \mu \pi(\widetilde{\phi}_{\Conj{x_n}}(\omega)) \mu^{-1} = \pi(\mu \cdot \widetilde{\phi}_{\Conj{x_n}}(\omega)),\]
where the first two equalities follow from the fact that $\pi \circ \widetilde{\phi}_{\Conj{x_n}} = \text{id}$ and
the third follows from the first conclusion of Proposition \ref{proposition:autaction}.

We now verify \eqref{eqn:nuindep1} as follows:
\[\nu_1 \cdot \widetilde{\phi}_{\Conj{x_n}}(\nu_1^{-1} \eta \nu_1) = \nu_1 \cdot \widetilde{\phi}_{\Conj{x_n}}(\mu \omega \mu^{-1}) = \nu_1 \mu \cdot \widetilde{\phi}_{\Conj{x_n}}(\omega) = \nu_2 \cdot \widetilde{\phi}_{\Conj{x_n}}(\nu_2^{-1} \eta \nu_2). \qedhere\]
\end{proof}

This completes the construction of $\widetilde{\phi}$.

\paragraph{Some naturality properties.}
Before we study the edge and conjugation relators, we first need to verify the following two
naturality properties of $\widetilde{\phi}$.  Starting now we will
use the following notation which was introduced in \S \ref{section:thepresentation}: for a vertex
$\Conj{z}$ of $\ABases_n$ and $\eta \in \IA_n$ satisfying $\eta(\Conj{z}) = \Conj{z}$, we will denote
$\eta$ considered as an element of
\[(\IA_n)_{\Conj{z}} < \BigFreeProd_{\Conj{z} \in (\ABases_n)^{(0)}} (\IA_n)_{\Conj{z}}\]
by $\eta_{\Conj{z}}$.

\begin{claims}
\label{claim:respectgen}
The following two identities hold.
\begin{compactitem}
\item Let $1 \leq a,b \leq n$ be distinct and $1 \leq i \leq n$ be arbitrary.  Then
$\widetilde{\phi}((\Con{x_a}{x_b})_{\Conj{x_{i}}}) = \QCon{x_a}{x_b}$.
\item Let $1 \leq a,b,c \leq n$ be distinct, let $\alpha,\beta,\gamma \in \{\pm 1\}$ be arbitrary, and
let $1 \leq i \leq n$ be such that $i \neq a$.  Then
$\widetilde{\phi}((\Mulcomm{x_a^{\alpha}}{x_b^{\beta}}{x_c^{\gamma}})_{\Conj{x_i}}) = \QMulcomm{x_a^{\alpha}}{x_b^{\beta}}{x_c^{\gamma}}$.
\end{compactitem}
\end{claims}
\begin{proof}[Proof of claim]
The proofs of two two identities are similar; we will deal with the first and leave the second to the reader.
It is clear from the construction that $\widetilde{\phi}((\Con{x_a}{x_b})_{\Conj{x_{n}}}) = \QCon{x_a}{x_b}$.
For $1 \leq i < n$, we have $P_{i,n}(x_n) = x_i$, and thus by definition we have
\begin{align*}
\widetilde{\phi}((\Con{x_a}{x_b})_{\Conj{x_{i}}}) &= P_{i,n} \cdot \widetilde{\phi}_{\Conj{x_n}}(P_{i,n}^{-1} \Con{x_a}{x_b} P_{i,n}) = P_{i,n} \cdot \widetilde{\phi}_{\Conj{x_n}}(\Con{P_{i,n}^{-1}(x_a)}{P_{i,n}^{-1}(x_b)}) \\
&= P_{i,n} \cdot \QCon{P_{i,n}^{-1}(x_a)}{P_{i,n}^{-1}(x_b)} = \QCon{x_a}{x_b};
\end{align*}
here the last equality follows from the first part of Proposition \ref{proposition:autaction} and the definition
of the endomorphisms in \S \ref{section:thepresentation}.
\end{proof}

\begin{claims}
\label{claim:partialinverse}
Let $\Conj{z}$ be a vertex of $\ABases_n$.  Then for $\eta \in (\IA_n)_{\Conj{z}}$ we have
$\pi(\widetilde{\phi}_{\Conj{z}}(\eta)) = \eta$.
\end{claims}
\begin{proof}[Proof of claim]
Pick $\nu \in \Aut(F_n)$ such that $\nu(x_n) = z$.  Then
\[\pi(\widetilde{\phi}_{\Conj{z}}(\eta)) = \pi(\nu \cdot \widetilde{\phi}_{\Conj{x_n}}(\nu^{-1} \eta \nu)) = \nu \pi(\widetilde{\phi}_{\Conj{x_n}}(\nu^{-1} \eta \nu)) \nu^{-1} = \nu \nu^{-1} \eta \nu \nu^{-1} = \eta;\]
here the second equality uses the first part of Proposition \ref{proposition:autaction}.
\end{proof}

\paragraph{The edge and conjugation relators.}
We now check that $\widetilde{\phi}$ takes the edge and conjugation relators to $1$.

\begin{claims}[Edge relators]
\label{claim:edge}
If $e$ is an edge of $\ABases_n$ with endpoints $\Conj{z}$ and $\Conj{z'}$ and $\eta \in (\IA_n)_e$, then
$\widetilde{\phi}(\eta_{\Conj{z}} \eta_{\Conj{z'}}^{-1}) = 1$.
\end{claims}
\begin{proof}[Proof of claim]
We first consider the special case where $z=x_n$ and $z' = x_{n-1}$.  
Proposition~\ref{proposition:stabilizergenerators}, with $k=2$, states that 
$(\IA_n)_{\Conj{x_{n-1}},\Conj{x_n}}$ is generated by
\begin{equation}
\label{eqn:stabgen}
\Set{$\Con{x_a}{x_b}$}{$1 \leq a,b \leq n$ distinct} \cup \Set{$\Mulcomm{x_a}{x_b}{x_c}$}{$1 \leq a,b,c \leq n$ distinct, $a \neq n-1,n$}.
\end{equation}
Claim \ref{claim:respectgen} implies that for all elements $\omega$ in \eqref{eqn:stabgen}, we have
$\widetilde{\phi}(\omega_{\Conj{x_{n-1}}}) = \widetilde{\phi}(\omega_{\Conj{x_n}})$.  It follows that for all
$\eta \in (\IA_n)_{\Conj{x_{n-1}},\Conj{x_n}}$ we have 
$\widetilde{\phi}(\eta_{\Conj{x_{n-1}}}) = \widetilde{\phi}(\eta_{\Conj{x_n}})$, as desired.

We now turn to general edges $e$ with endpoints $\Conj{z}$ and $\Conj{z'}$ and $\eta \in (\IA_n)_e$.  
There exists some $\nu \in \Aut(F_n)$ such that $\nu(x_n) = z$ and $\nu(x_{n-1}) = z'$, and hence
$\nu P_{n-1,n}(x_{n}) = z'$.  Setting $\eta' = \nu^{-1} \eta \nu \in (\IA_n)_{\Conj{x_{n-1}},\Conj{x_n}}$,
we have
\[\widetilde{\phi}(\eta_{\Conj{z}}) = \nu \cdot \widetilde{\phi}_{\Conj{x_n}}(\nu^{-1} \eta \nu) = \nu \cdot \widetilde{\phi}_{\Conj{x_n}}(\eta')\]
and
\[\widetilde{\phi}(\eta_{\Conj{z'}}) = \nu P_{n-1,n} \widetilde{\phi}_{\Conj{x_n}}(P_{n-1,n}^{-1} \nu^{-1} \eta \nu P_{n-1,n}) = \nu \cdot \widetilde{\phi}_{\Conj{x_{n-1}}}(\eta').\]
By the previous paragraph, we have $\widetilde{\phi}_{\Conj{x_n}}(\eta') = \widetilde{\phi}_{\Conj{x_{n-1}}}(\eta')$, so
we conclude that $\widetilde{\phi}(\eta_{\Conj{z}}) = \widetilde{\phi}(\eta_{\Conj{z'}})$, as desired.
\end{proof}

\begin{claims}[Conjugation relators]
\label{claim:conjugation}
If $\Conj{z}$ and $\Conj{z'}$ are vertices of $\ABases_n$ and $\eta \in (\IA_n)_{\Conj{z}}$ and
$\omega \in (\IA_n)_{\Conj{z'}}$, then 
$\widetilde{\phi}(\omega_{\Conj{z'}} \eta_{\Conj{z}} \omega_{\Conj{z'}}^{-1} (\omega \eta \omega^{-1})_{\Conj{\omega(z)}})=1.$
\end{claims}
\begin{proof}[Proof of claim]
Choose $\nu \in \Aut(F_n)$ such that $\nu(x_n) = z$.  We then have
\begin{align*}
\widetilde{\phi}(\omega_{\Conj{z'}} \eta_{\Conj{z}} \omega_{\Conj{z'}}^{-1}) &= \widetilde{\phi}_{\Conj{z'}}(\omega) \widetilde{\phi}_{\Conj{z}}(\eta) \widetilde{\phi}_{\Conj{z'}}(\omega)^{-1} = \pi(\widetilde{\phi}_{\Conj{z'}}(\omega)) \cdot \widetilde{\phi}_{\Conj{z}}(\eta) = \omega \cdot \widetilde{\phi}_{\Conj{z}}(\eta) \\
&= \omega \nu \cdot \widetilde{\phi}_{\Conj{x_n}}(\nu^{-1} \eta \nu) = \omega \nu \cdot \widetilde{\phi}_{\Conj{x_n}}((\omega \nu)^{-1} \omega \eta \omega^{-1} (\omega \nu)) = \widetilde{\phi}((\omega \eta \omega^{-1})_{\Conj{\omega(z)}}),
\end{align*}
as desired.  The second equality follows from the third part of Proposition \ref{proposition:autaction}, the
third equality follows from Claim \ref{claim:partialinverse}, and the remainder of the equalities are straightforward
applications of the definitions.
\end{proof}

Claims \ref{claim:edge} and \ref{claim:conjugation} imply that $\widetilde{\phi}$ descends to a homomorphism
$\phi\colon \IA_n \rightarrow \Quotient_n$.

\paragraph{We have an inverse.}
To complete the proof, it remains to prove the following.

\begin{claims}
We have $\phi \circ \pi = \text{id}$.
\end{claims}
\begin{proof}[Proof of claim]
Claim \ref{claim:respectgen} implies that this holds for the generators of $\Quotient_n$.
\end{proof}

This completes the proof of Theorem \ref{maintheorem:finitelpres}.
\end{proof}

\section{Computations for the L-presentation}
\label{section:somecomputations}

This section contains the postponed proofs of Propositions \ref{proposition:verifyrels}, \ref{proposition:verifyendomorphisms},
\ref{proposition:autaction}, and~\ref{proposition:mapfromkernel}.  These proofs are done with the
aid of a computer.  We will discuss our computational framework in \S \ref{section:computationalframework} and
then prove the propositions in \S \ref{section:maptoIAn} -- \ref{section:mapcomputations}.

We will use the following notation throughout the rest of the paper.  Let
$S_{\Aut}(n)^{\ast}$ denote the free monoid on the set $S_{\Aut}(n)$.  In \S \ref{section:thepresentation},
we defined a function $\theta\colon S_{\Aut}(n) \rightarrow \End(F(S_{\IA}(n)))$.  This
naturally extends to a function $\theta\colon S_{\Aut}(n)^{\ast} \rightarrow \End(F(S_{\IA}(n)))$.

\subsection{Computational framework}
\label{section:computationalframework}

As we discussed in the introduction, 
we use the GAP System to mechanically verify the large number of equations we have to check.
These verifications are in the file \verb+h2ia.g+, distributed with this document and also available at the authors' websites.

We use GAP's built-in functionality to model $F_n$ as a free group on the eight generators \verb+xa+, \verb+xb+, \verb+xc+, \verb+xd+, \verb+xe+, \verb+xf+, \verb+xg+, and \verb+y+.
Since our computations never involve more than $8$ variables, computations in this group suffice to show that our computations hold in general.

Elements of the sets $S_{\Aut}(n)$ and $S_{\IA}(n)$ are parametrized over basis elements from $F_n$ and their inverses, so we model these sets using lists.
For example, we model the generator $\Mul{x_a}{x_b}$ as the list \verb+["M",xa,xb]+, $\Con{y}{x_a}$ as \verb+["C",y,xa]+, and $\Mulcomm{x_a^{-1}}{y}{x_c}$ as \verb+["Mc",xa^-1,y,xc]+.
We model $P_{a,b}$ as \verb+["P",xa,xb]+ and $I_a$ as \verb+["I",xa]+.
The examples should make clear: the first entry in the list is a string key \verb+"M"+, \verb+"C"+, \verb+"Mc"+, \verb+"P"+, or \verb+"I"+, indicating whether the list represents a transvection, conjugation move, commutator transvection, swap or inversion.
The parameters given as subscripts in the generator are then the remaining elements of the list, in the same order.

GAP's built-in free group functionality expects the basis elements to be variables, not lists, so we do not use it to model $S_{\Aut}(n)^{\ast}$ and $F(S_{\IA}(n))$.
We model inverses of generators as follows:
the inverse of \verb+["M",xa,xb]+ is \verb+["M",xa,xb^-1]+ and the inverse of \verb+["C",xa,xb]+ is  \verb+["C",xa,xb^-1]+, but the inverse of \verb+["Mc",xa,xb,xc]+ is \verb+["Mc",xa,xc,xb]+.
Swaps and inversions are their own inverses.
Technically, this means that we are not really modeling $S_{\Aut}(n)^{\ast}$ and $F(S_{\IA}(n))$; instead we model structures where the order relations for swaps and inversions and the relation R0 for inverting commutator transvections are built in.
This is not a problem because our verifications always show that certain formulas are trivial modulo our relations, and we can always apply the R0 and order relations as needed.

We model words in $S_{\Aut}(n)^{\ast}$ and $F(S_{\IA}(n))$ as lists of generators and inverse generators.
The empty word \verb+[]+ represents the trivial element.
We wrote several functions in \verb+h2ia.g+ that perform common tasks on words.
The function \verb+pw+ takes any number of words (reduced or not) as arguments and returns the freely reduced product of those  words in the given order, as a single word.
The function \verb+iw+ inverts its input word and the function \verb+cyw+ cyclically permutes its input word.

The function \verb+iarel+ outputs the relations $R_{\IA}^0(n)$.
We introduce some extra relations for convenience.
The function \verb+exiarel+ outputs these extra relations and the code generating the list \verb+exiarelchecklist+ derives the extra relations from the basic relations.
The function \verb+theta+ takes in a word $w$ in $S_{\Aut}(n)$ and a word  $v$ in $S_{\IA}(n)$, and returns $\ttheta(w)(v)$.
In addition to the functions described here, we often define simple macros for carrying out the verifications.

The function \verb+applyrels+ is particularly useful, because it inserts multiple relations into a word.
It takes in two inputs: a starting word and a list of words with placement indicators.
The function recursively inserts the first word from the list in the starting word at the given position, reduces the word, and then calls itself with the new word as the starting word and with the same list of insertions, with the first dropped.

For example, the following command appears in the justification of Proposition~\ref{proposition:autaction}:
\begin{verbatim}
applyrels(
     pw(
          theta([["M",xa,xb],["M",xa,xb^-1]], [["Mc",xb,xa,xe]]),
          [["Mc",xb,xe,xa]]
     ),
     [
          [6,iarel(5,[xa,xb,xe])],
          [6,iw(exiarel(1,[xb^-1,xe,xa]))],
          [2,iw(iarel(4,[xe,xa,xb]))],
          [2,iarel(6,[xb^-1,xe^-1,xa])],
          [2,iw(iarel(5,[xb,xe,xa]))]
     ]
)
\end{verbatim}
This tells the GAP system to compute the effect of $\ttheta(\Mul{x_a}{x_b}\Mul{x_a}{x_b}^{-1})$ on $\Mulcomm{x_b}{x_a}{x_e}$.
Then it multiplies this by $\Mulcomm{x_b}{x_e}{x_a}$, the inverse of $\Mulcomm{x_b}{x_a}{x_e}$.
It then freely reduces this word.
The system inserts a version of R5 after the sixth letter in this word, and reduces the result to a new word.
Then it inserts the inverse of one of the extra relations after the sixth letter in the new word and reduces it.
It continues with inserting relations and reducing the resulting expressions, inserting instances of R4, R5 and R6.
Since the entire expression evaluates to \verb+[]+, we have expressed 
\[ \ttheta(\Mul{x_a}{x_b}\Mul{x_a}{x_b}^{-1})(\Mulcomm{x_b}{x_a}{x_e})\cdot \Mulcomm{x_b}{x_e}{x_a}\]
as a product of relations in $\Quotient_n$.
In any example like this, an interested reader can reproduce our reduction process by removing all the list entries from the second input of the \verb+applyrels+ call, and then adding them back in one at a time, evaluating after each one.

\subsection{Verifying the map to \texorpdfstring{$\IA_n$}{IAn}}
\label{section:maptoIAn}
First we prove Proposition~\ref{proposition:verifyrels}, which states that our relations $R_{\IA}^0(n)$ hold in $\IA_n$.
\begin{proof}[Proof of Proposition~\ref{proposition:verifyrels}]
The code generating the list \verb+verifyiarel+ generates examples of all the relations in $R_{\IA}^0(n)$, with all allowable configurations of coincidences between the subscripts on the generators.
It converts each of these relations into automorphisms of $F_n$ and evaluates them on a basis for $F_n$, returning \verb+true+ if all basis elements are unchanged.
We evaluate on a fixed finite-rank free group, but since the basic relations involve at most six generators, those evaluations suffice to show the result in general.
Since \verb+verifyiarel+ evaluates to a list of \verb+true+, this means that all these relations are true.
\end{proof}

Next we prove Proposition~\ref{proposition:verifyendomorphisms}, which states that $\ttheta$ acts by conjugation when evaluated on generators (see equation~\eqref{eqn:thetakey}).
\begin{proof}[Proof of Proposition~\ref{proposition:verifyendomorphisms}]
The code generating the list \verb+thetavsconjaut+ goes through all possible configurations for a pair of generators $s$ from $S_{\Aut}(n)$ and $t$ from $S_{\IA}(n)$, evaluates $\ttheta(s)(t)$ as a product of generators, and then evaluates both $\ttheta(s)(t)$ and $sts^{-1}$ on a basis for $F_n$.
It returns \verb+true+ when both have the same effect on all basis elements.
Since \verb+thetavsconjaut+ evaluates to a list of  \verb+true+, the proposition holds.
\end{proof}

\subsection{Verifying Proposition~\ref{proposition:autaction}}
\label{section:actioncomputations}

\begin{proof}[Proof of Proposition~\ref{proposition:autaction}]
The action of $\Aut(F_n)$ on $\Quotient_n$ is given by our substitution rule endomorphism map
\[\ttheta\co S_{\Aut}(n)\to \End(F(S_{\IA}(n))).\]

First of all, it is clear that for each $s\in S_{\Aut}(n)$, the element
$\ttheta(s)$ defines an endomorphism of $\Quotient_n$.
This is because the subgroup of $F(S_{\IA}(n))$ normally generated by the relations of $\Quotient_n$ is invariant under $\ttheta(s)$ by the definition of $\Quotient_n$.

Next, we verify that $\ttheta(s)$ is an automorphism of $\Quotient_n$.
If $s$ is a swap or an inversion, then it is clear from the definition of $\ttheta$ that this is the case.
In the code generating the list \verb+thetainverselist+, we compute $\ttheta(s)(\ttheta(s^{-1})(t))t^{-1}$ for $s=\Mul{x_a}{x_b}$ and for all possible configurations of $t$ relative to $s$.
In each case, we reduce it to the trivial word using relations for $\Quotient_n$.
It is not hard to deduce that $\ttheta(s)(\ttheta(s^{-1})(t))=t$ in $\Quotient_n$ for the remaining choices of $s=\Mul{x_a}{x_b}^{-1}$, $\Mul{x_a^{-1}}{x_b}$ and $\Mul{x_a^{-1}}{x_b}^{-1}$, using the fact that it is true for $s=\Mul{x_a}{x_b}$.

So this shows that $\ttheta$ defines an action
\[F(S_{\Aut}(n))\to \Aut(\Quotient_n).\]
Now we need to verify that this action descends to an action of $\Aut(F_n)$.
To show this, it is enough to show that for every relation $r$ in a presentation for $\Aut(F_n)$, we have 
\begin{equation}\label{eqn:actbyrelator}
\ttheta(r)(t)=t\quad \text{ in $\Quotient_n$},
\end{equation}
for $t$ taken from a generating set for $\Quotient_n$.
To check this, we use the same version of Nielsen's presentation for $\Aut(F_n)$ that we used in~\cite[Theorem~5.5]{DayPutmanBirmanIA}.
The generators are the same set $S_{\Aut}(n)$ we use here, and the relations fall into four classes N1--N5.
N1 are sufficient relations for the subgroup generated by swaps and inversions and N2 are relations indicating how to conjugate transvections by swaps and inversions.
It is an exercise to see that~\eqref{eqn:actbyrelator} holds for relations of class N1 and N2.
Relations N3, N4 and N5 are more complicated relations.
For each of these, we compute $\ttheta(r)(t)t^{-1}$ on generators $t$ (for $t$ with enough configurations of subscripts to include a generating set) and reduce the resulting expressions to $1$ using relations from $\Quotient_n$.
These computations are given in the code generating the lists \verb+thetaN3list+, \verb+thetaN4list+, and \verb+thetaN5list+.
Since these evaluate to lists of the trivial word, this verifies equation~\eqref{eqn:actbyrelator}.
We have shown that the action of an element of $\Aut(F_n)$ on $\Quotient_n$ does not depend on the word in $F(S_{\Aut}(n))$ we use to represent it.

Since we have shown that $\ttheta$ defines an action, now we can check the three properties asserted in Proposition~\ref{proposition:autaction}.
We have already verified the first point (we took the definition of the action to agree with it).
To verify the second point, we need to check that
for $w,s\in S_{\IA}(n)$, there is $\widetilde w\in F(S_{\Aut}(n))$ representing $w$ 
with
\begin{equation}\label{eq:conjcheck}
\ttheta(\widetilde w)(s)=wsw^{-1}\quad\text{ in $\Quotient_n$}.
\end{equation}
In fact, it is enough to verify this for $w,s$ in a smaller generating set, and the generating set that $s$ is taken from may depend on $w$.
In the code generating \verb+thetaconjrellist+, for each choice of $w$ from $S_{\Mag}(n)$, we lift $w$ to $\widetilde w\in F(S_{\Aut}(n))$, and for several configurations of subscripts in the generator $s$, we reduce $\ttheta(\widetilde w)(s)ws^{-1}w^{-1}$ to the identity using relations from $\Quotient_n$.
We use enough configurations of subscripts in $s$ to cover all cases for $s$ in a generating set (a conjugate of $S_{\Mag}(n)$).

To check the third point, we use the generating set $(S_{\Aut}(n))_{\Conj{x_n}}$ for $(\Aut(F_n))_{\Conj{x_n}}$ mentioned in the proof of Proposition~\ref{proposition:stabilizergenerators} above, namely
\[
\begin{split}
&\Set{$\Mul{x_a^\alpha}{x_b}$}{$1\leq a\leq n-1$,$1\leq b\leq n$, $\alpha=\pm1$, $a\neq b$}
\cup \Set{$P_{a,b}$}{$1\leq a,b\leq n$, $a\neq b$, $a\neq n$}
 \\
&\cup \Set{$I_a$}{$1\leq a\leq n$} \cup \Set{$\Con{x_n}{x_a}$}{$1\leq a\leq n-1$}.
\end{split}
\]
We need to check that for each of these generators, there is $w\in F(S_{\Aut}(n))$ representing it with $\ttheta(w)(s)$ in $(\Quotient_n)_{\Conj{x_n}}$ (really, that $\ttheta(w)(s)$ is equal in $\Quotient_n$ to  an element of $(\Quotient_n)_{\Conj{x_n}}$).
This is clear from the definition of $\ttheta$ for $w$ a swap or inversion.
It can be verified for $w=\Mul{x_a^\alpha}{x_b}$ by inspecting Table~\ref{table:thetadef}.
For $w$ representing $\Con{x_n}{x_a}$, the fact that $\ttheta(w)(s)\in (\Quotient_n)_{\Conj{x_n}}$ follows from the second point in this proposition, since $\Con{x_n}{x_a}\in \IA_n$.
\end{proof}

\subsection{Verifying Proposition~\ref{proposition:mapfromkernel}}
\label{section:mapcomputations}

Here we prove Proposition~\ref{proposition:mapfromkernel}.
We recall the statement: $\BirKer_{n-1,1}$ is the kernel of the natural map $(\IA_n)_{\Conj{x_n}}\to \IA_{n-1}$, and the proposition asserts that the inclusion $\BirKer_{n-1,1}\hookrightarrow \IA_n$ factors as the composition of a map $\BirKer_{n-1,1}\to \Quotient_n$ with the projection $\Quotient_n\to\IA_n$.

The proof uses the finite L-presentation for $\BirKer_{n-1,1}$ from~\cite{DayPutmanBirmanIA}.
We note that \cite[Theorem~D]{DayPutmanBirmanIA} asserts the existence of such a presentation, and \cite[Theorem~6.2]{DayPutmanBirmanIA} gives the precise statement that we use in the computations.
Since this L-presentation is in fact a presentation for $\BirKer_{n-1,1}$, we use the same notation for $\BirKer_{n-1,1}$ as a subset of $\IA_n$ and $\BirKer_{n-1,1}$ as the group given by this presentation.

We do not reproduce the L-presentation here, but instead we describe some of its features.
Its finite generating set is
\[
\begin{split}
S_{\mathcal K}(n)=&
\Set{$\Mulcomm{x_a^\alpha}{x_n^\epsilon}{x_b^\beta}$}{$1\leq a,b\leq n-1$, $a\neq b$, $\alpha,\beta,\epsilon\in\{1,-1\}$}
\\
&
\cup
\Set{$\Con{x_n}{x_a}$}{$1\leq a \leq n-1$}
\cup
\Set{$\Con{x_a}{x_n}$}{$1\leq a \leq n-1$}
\end{split}.
\]
The substitution endomorphisms of the L-presentation for $\BirKer_{n-1,1}$ are indexed by a 
finite generating set $(S_{\Aut}(n))_{\Conj{x_n}}$ for $(\Aut(F_n))_{\Conj{x_n}}$.
The endomorphisms themselves are the image of a map
\[\tphi\co (S_{\Aut}(n))_{\Conj{x_n}}\to\End(F(S_{\mathcal K}(n))).\]

\begin{proof}[Proof of Proposition~\ref{proposition:mapfromkernel}]
Since $S_{\mathcal K}(n)$ is a subset of $S_{\IA}(n)$, we map 
$\BirKer_{n-1,1}$ to $\Quotient_n$ by sending each generator to the generator of the same name.
To verify that this map on generators extends to a well defined map of groups, we need to check that each defining relation from $\BirKer_{n-1,1}$ maps to the trivial element of $\Quotient_n$.
Since $\BirKer_{n-1,1}$ is given by a L-presentation, we proceed as follows:
\begin{enumerate}
\item We check that each of the basic relations from $\BirKer_{n-1,1}$  maps to the trivial element of $\Quotient_n$.
\item 
We check that for $s\in(S_{\Aut}(n))_{\Conj{x_n}}$ and $t\in S_{\mathcal K}(n)$,
we have
\[\tphi(s)(t)=\ttheta(s)(t) \quad\text{ in $\Quotient_n$},\]
where we use $(S_{\Aut}(n))_{\Conj{x_n}}\subset S_{\Aut}(n)$ to plug $s$ into $\ttheta$, and we interpret both expressions in $\Quotient_n$ using $F(S_{\mathcal K}(n))\subset F(S_{\IA}(n))$.
\end{enumerate}
The first point is verified in the code generating the list \verb+kfromialist+.
The function \verb+krel+ produces the basic relations from $\BirKer_{n-1,1}$, and we reduce each relation to the identity by applying relations from $\Quotient_n$.
The second point is verified in the code generating the list \verb+thetavsphlist+.
For each choice of pairs of generators, we reduce the difference of $\ttheta$ and $\tphi$ using relations from $\Quotient_n$.

With these two points verified, one can easily check by induction that every extended relation (starting with a basic relation, and applying any sequence of rewriting rules) maps to the identity element in $\Quotient_n$.
\end{proof}

\section{Generators for \texorpdfstring{$\HH_2(\IA_n)$}{H2(IAn)}}
\label{section:h2gen}

In this section, we prove Theorem \ref{maintheorem:h2finite}, which asserts that
there exists a finite subset of $\HH_2(\IA_n)$ whose $\GL_n(\Z)$-orbit spans
$\HH_2(\IA_n)$.  In fact, we gave an explicit list of generators in Table~\ref{table:commutatorrelatorssmall};
each generator is of the form $\Relh{r}$ for a commutator relation $r$.
This list is reproduced in Table \ref{table:commutatorrelators}, which also introduces
the notation $h_i(\dotsc) \in \HH_2(\IA_n)$ for the associated elements of homology (this notation
will be used during the calculations in \S \ref{section:finalcalcs}, though we will not use
it in this section).  
The following theorem asserts that this list is complete; it is a more precise
form of Theorem \ref{maintheorem:h2finite} and will be the main result of this section.

\begin{theorem}
\label{theorem:h2gen}
Fix $n \geq 2$.  Let $S_H(n)$ be the set of commutator relators in Table \ref{table:commutatorrelators}.
Then the $\GL_n(\Z)$-orbit of the set $\Set{$\Relh{r}$}{$r \in S_H(n)$}$ spans $\HH_2(\IA_n)$.
\end{theorem}

Before proving Theorem \ref{theorem:h2gen}, we will use it to derive Theorem \ref{maintheorem:surjectivestability}.
\begin{proof}[{Proof of Theorem \ref{maintheorem:surjectivestability}}]
Recall that this theorem asserts that for $n \geq 6$, the $\GL_{n+1}(\Z)$-orbit of the image of the natural map $\HH_2(\IA_n) \rightarrow \HH_2(\IA_{n+1})$
spans $\HH_2(\IA_{n+1})$.
Let $S_{n+1} \subset \GL_{n+1}(\Z)$ be the subgroup consisting of permutation matrices.
By inspecting Table~\ref{table:commutatorrelators}, it is clear that
the $S_{n+1}$-orbit of the image of $\Set{$\Relh{r}$}{$r \in S_H(n)$} \subset \HH_2(\IA_n)$ in $\HH_2(\IA_{n+1})$
is $\Set{$\Relh{r}$}{$r \in S_H(n+1)$}$.
This uses the fact that $n\geq 6$, since the commutator relations in $S_H(n)$ use generators involving at most six basis elements.
\end{proof}

We now turn to the proof of Theorem \ref{theorem:h2gen}.  We start by introducing some notation.
Let $F=F(S_{\IA}(n))$ and let $R \subset F$ denote the full set of relations of $\IA_n$, so $\IA_n=F/R$.
Define $\HIAT = R/[F,R]$, and for $r \in R$ denote by $\Elt{r}$ the associated element of $\HIAT$.
There is a natural map $\HIAT \rightarrow F^{\text{ab}}$, and the starting point for our proof
is the following lemma.  In it, recall from the beginning of \S \ref{section:somecomputations} that
$S_{\Aut}(n)^{\ast}$ is the free monoid on the set $S_{\Aut}(n)$.

\begin{lemma}
\label{lemma:hiat}
The group $\HIAT$ is an abelian group which is generated by
\[\Set{$\Elt{\ttheta(w)(r)}$}{$w \in S_{\Aut}(n)^{\ast}$ and $r$ is one of the relations R0--R9 from Table \ref{table:basicrelsian}}.\]
Also, we have $\HH_2(\IA_n) = \Ker(\HIAT \rightarrow F^{\text{ab}})$.
\end{lemma}
\begin{proof}
The group $\HIAT$ is abelian since $[R,R] \subset [F,R]$.  For $v \in F$ and $r \in R$, we have $[v,r] \in [F,R]$,
so $\Elt{v r v^{-1}} = \Elt{r}$.  The indicated generating set for $\HIAT$ thus follows from
Theorem \ref{maintheorem:finitelpres}.  As for the final statement of the lemma, we follow one
of the standard proofs of Hopf's formula \cite{BrownCohomology}.  The 5-term exact sequence in group
homology associated to the short exact sequence
\[0 \longrightarrow R \longrightarrow F \longrightarrow \IA_n \longrightarrow 0\]
is
\[\HH_2(F) \longrightarrow \HH_2(\IA_n) \longrightarrow R/[F,R] \longrightarrow \HH_1(F) \longrightarrow \HH_1(\IA_n) \longrightarrow 0.\]
Since $F$ is free, we have $\HH_2(F) = 0$, and the claim follows.
\end{proof}

Our goal will be to take an element of $\HIAT$ that happens to lie in $\HH_2(\IA_n)$ and
rewrite it as a sum of elements of the form
\begin{equation}
\label{eqn:goodset}
\Set{$\Elt{\ttheta(w)(r)}$}{$w \in S_{\Aut}(n)^{\ast}$ and $r$ is one of the relations H1--H9 from Table \ref{table:commutatorrelators}}.
\end{equation}
The relations H1--H4 are the same as R1--R4, and the relations H5--H7 are the same as R6--R9.  The troublesome
relations are R0, R5, and R6, none of which lie in $\HH_2(\IA_n)$.  For $r \in R$, we have
$\Elt{r} \in \HH_2(\IA_n)$ if and only if the exponent-sum of each generator in $S_{\IA}(n)$ appearing in it is $0$.
For our problematic relations R0, R5, and R6, the exponent-sum of all the conjugations moves is already $0$,
so we will only need to study the exponent-sums of the commutator transvections.

%%%%%%%%%%%%%%%%%%%%%%%%%%%%%%%%%%%%%%%%%%%%%%%%%%%%%%%%%%%%%%%%%%%%%%%%%%%%%%%%%%%
%%%%%%%%%%%%%%%%%%%%%%%%%%%%%%%%%%%%%%%%%%%%%%%%%%%%%%%%%%%%%%%%%%%%%%%%%%%%%%%%%%%
%%%%%%%%%%%%%%%%%%%%%%%%%%%%%%%%%%%%%%%%%%%%%%%%%%%%%%%%%%%%%%%%%%%%%%%%%%%%%%%%%%%
\begin{table}[p!]
\begin{tabular}{p{0.95\textwidth}}
\toprule
\begin{center}
\textbf{Basic commutator relators for $\IA_n$}
\end{center}
\vspace*{-0ex}\begin{enumerate}\setlength{\itemsep}{0ex}
\item[H1.]
$h_1(x_a^\alpha,x_b^\beta,x_c^\gamma,x_d^\delta)$:
possibly with $b=d$,
\[ [\Con{x_a}{x_b}^\beta,\Con{x_c}{x_d}^\delta]=1.\]
\item[H2.]
$h_2(x_a^\alpha,x_b^\beta,x_c^\gamma,x_d^\delta,x_e^\epsilon,x_f^\zeta)$:
possibly with $\{b,c\}\cap\{e,f\}\neq \varnothing$ or with $x_a^{\alpha}=x_d^{-\delta}$, as long as $x_a^\alpha\neq x_d^\delta$, $a\notin\{e,f\}$ and $d\notin \{b,c\}$,
\[ [\Mulcomm{x_a^\alpha}{x_b^\beta}{x_c^\gamma},\Mulcomm{x_d^\delta}{x_e^\epsilon}{x_f^\zeta}]=1.\]
\item[H3.]
$h_3(x_c^\gamma,x_d^\delta,x_e^\epsilon, x_a^\alpha,x_b^\beta)$:
possibly with $b\in\{d,e\}$, if $c\notin\{a,b\}$ and $a\notin\{c,d,e\}$,
\[ [\Con{x_a}{x_b}^\beta,\Mulcomm{x_c^\gamma}{x_d^\delta}{x_e^\epsilon}]=1.\]
\item[H4.]
$h_4(x_a^\alpha,x_b^\beta,x_c^\gamma)$:
\[[\Con{x_c}{x_b}^\beta\Con{x_a}{x_b}^\beta,\Con{x_c}{x_a}^\alpha]=1.\]
\item[H5.]
$h_5(x_a^\alpha,x_b^\beta,x_c^\gamma,x_d^\delta)$:
\[[\Con{x_a}{x_c}^{-\gamma},\Con{x_a}{x_d}^{-\delta}][\Con{x_a}{x_b}^{-\beta},\Mulcomm{x_b^\beta}{x_c^\gamma}{x_d^\delta}]=1.\]
\item[H6.]
$h_6(x_a^\alpha,x_b^\beta,x_c^\gamma,x_d^\delta,x_e^\epsilon)$:
possibly with $b=e$ or $c=e$,
\[
[\Mulcomm{x_a^\alpha}{x_b^\beta}{x_c^\gamma},\Mulcomm{x_d^\delta}{x_a^\alpha}{x_e^\epsilon}]
[\Mulcomm{x_d^\delta}{x_a^\alpha}{x_e^\epsilon},\Mulcomm{x_d^\delta}{x_c^\gamma}{x_b^\beta}]
[\Mulcomm{x_d^\delta}{x_c^\gamma}{x_b^\beta},\Con{x_d}{x_e}^{-\epsilon}]
=1.
\]
\item[H7.]
$h_7(x_a^\alpha,x_b^\beta,x_c^\gamma,x_d^\delta)$:
possibly with $b=d$,
\[
[\Mulcomm{x_c^\gamma}{x_a^\alpha}{x_d^\delta},\Con{x_a}{x_b}^\beta]
[\Con{x_c}{x_d}^{-\delta},\Mulcomm{x_c^\gamma}{x_a^\alpha}{x_b^\beta}]
[\Mulcomm{x_c^\gamma}{x_a^\alpha}{x_b^\beta},\Mulcomm{x_c^\gamma}{x_a^\alpha}{x_d^\delta}]
=1.
\]
\item[H8.]
$h_8(x_a^\alpha,x_b^\beta,x_c^\gamma,x_d^\delta)$:
\[
[\Mulcomm{x_a^\alpha}{x_b^\beta}{x_c^\gamma},\Con{x_a}{x_d}^\delta\Con{x_b}{x_d}^\delta\Con{x_c}{x_d}^\delta]=1.
\]
\item[H9.]
$h_9(x_a^\alpha,x_b^\beta,x_c^\gamma)$:
\[
[\Con{x_a}{x_c}^\gamma\Con{x_b}{x_c}^\gamma,\Con{x_a}{x_b}^\beta\Con{x_c}{x_b}^\beta]
[\Mulcomm{x_a^\alpha}{x_b^\beta}{x_c^\gamma},\Con{x_b}{x_a}^{\alpha}\Con{x_c}{x_a}^\alpha]
=
1.\]
\end{enumerate}\\
\bottomrule
\end{tabular}
\caption{
The set $S_H(n)$ of
commutator relators such that the $\GL_n(\Z)$-orbit of
$\Set{$\Relh{r}$}{$r \in S_H(n)$}$ spans $\HH_2(\IA_n)$.
Distinct letters are assumed to represent distinct indices unless stated otherwise.
We give notation $h_i(\dotsc)$ for the elements in $\HH_2(\IA_n)$, which we use later.
}
\label{table:commutatorrelators}
\end{table}
%%%%%%%%%%%%%%%%%%%%%%%%%%%%%%%%%%%%%%%%%%%%%%%%%%%%%%%%%%%%%%%%%%%%%%%%%%%%%%%%%%%
%%%%%%%%%%%%%%%%%%%%%%%%%%%%%%%%%%%%%%%%%%%%%%%%%%%%%%%%%%%%%%%%%%%%%%%%%%%%%%%%%%%
%%%%%%%%%%%%%%%%%%%%%%%%%%%%%%%%%%%%%%%%%%%%%%%%%%%%%%%%%%%%%%%%%%%%%%%%%%%%%%%%%%%

We begin with the following lemma, which will allow us to mostly ignore our rewriting rules $\ttheta(\cdot)$.

\begin{lemma}
\label{lemma:rewritethetaR5R6}
Consider $w \in S_{\Aut}(n)^{\ast}$, and let $r$ be a relation of the form R0, R5, or R6.  Then
$\Elt{\ttheta(w)(r)} = h + h'$, where $h$ and $h'$ are as follows.
\begin{compactitem}
\item $h \in \HH_2(\IA_n)$ is a sum of elements from \eqref{eqn:goodset}.
\item $h'$ is a sum of elements of the form
\[\Set{$\Elt{r}$}{$r$ is one of the relations R0, R5, and R6 from Table \ref{table:basicrelsian}}.\]
\end{compactitem}
\end{lemma}
\begin{proof}
We can use induction to reduce to the case where $w = s \in S_{\Aut}(n)$.
The proof now is a combinatorial group-theoretic calculation: we will show how to rewrite
$\ttheta(s)(r)$ as a product of relations of the desired form.

We start by dealing with the case where $r$ is of the form R0.
Observe that $\ttheta(s)(\Mulcomm{x_a^\alpha}{x_b^\beta}{x_c^\gamma})$ agrees with $\ttheta(s)(\Mulcomm{x_a^\alpha}{x_c^\gamma}{x_b^\beta})^{-1}$ up to R0, except in two cases.
These are
\[\ttheta(\Mul{x_a^\alpha}{x_b}^\beta)(\Mulcomm{x_b^{-\beta}}{x_c^\gamma}{x_d^\delta})=\Mulcomm{x_a^{\alpha}}{x_c^\gamma}{x_d^\delta}\Mulcomm{x_b^{-\beta}}{x_c^\gamma}{x_d^\delta}\]
and
\[\ttheta(\Mul{x_a^\alpha}{x_b}^\beta)(\Mulcomm{x_b^{\beta}}{x_c^\gamma}{x_d^\delta})=\Con{x_a}{x_b}^\beta\Mulcomm{x_a^{-\alpha}}{x_c^\gamma}{x_d^\delta}\Mulcomm{x_b^\beta}{x_c^\gamma}{x_d^\delta}\Con{x_a}{x_b}^{-\beta}.\]
In the first case, 
\[\ttheta(\Mul{x_a^\alpha}{x_b}^\beta)(\Mulcomm{x_b^{-\beta}}{x_d^\delta}{x_c^\gamma})^{-1}=\Mulcomm{x_b^{-\beta}}{x_c^\gamma}{x_d^\delta}\Mulcomm{x_a^{\alpha}}{x_c^\gamma}{x_d^\delta}.\]
In the second case,
\[\ttheta(\Mul{x_a^\alpha}{x_b}^\beta)(\Mulcomm{x_b^{\beta}}{x_d^\delta}{x_c^\gamma})^{-1}=\Con{x_a}{x_b}^\beta\Mulcomm{x_b^\beta}{x_c^\gamma}{x_d^\delta}\Mulcomm{x_a^{-\alpha}}{x_c^\gamma}{x_d^\delta}\Con{x_a}{x_b}^{-\beta}.\]
In both cases, the two expressions differ by an application of R2.
This means that $\ttheta(s)$ of an R0 relation can always be written using R0 and R2 relations.

Next we explain the computations that prove the lemma for R5 and R6 relations.
These are in the list \verb+rewritethetaR5R6+.
In these computations we reduce $\ttheta(s)(r)$ to the trivial word where $s\in S_A^{\pm1}$ and $r$ is an R5 or R6 relation.
These reductions may use any of the basic relations for $\IA_n$, including R5 and R6 themselves, but notably may not use the images of R5 and R6 relations under $\ttheta$.
We may use the images of R1--R4, R7--R9, H8 and H9 under $\ttheta$.

Despite these restrictions, we may use the extra relations from \verb+exiarel+ in these computations. 
Our relation
\verb+exiarel(3,[xa,xb,xc,xd])+ is H8, and \verb+exiarel(4,[xa,xb,xc,xd])+ is equivalent to H8 modulo the basic relations.
Relation \verb+exiarel(7,[xa,xb,xc])+ is H9, and relation \verb+exiarel(6,[xa,xb,xc])+ and relation \verb+exiarel(8,[xa,xb,xc,xd])+ are equivalent to H9 modulo the basic relations.
All the other \verb+exiarel+ relations can be derived without using images of R5 or R6 under $\ttheta$.
These facts can be verified by inspecting \verb+exiarelchecklist+.

If $s$ is a swap or an inversion move, then acting by $\ttheta(s)$ is always the same as acting on the parameters of the relation by $s$ in the obvious way.
Therefore \verb+rewritethetaR5R6+ only contains cases where $s$ is a transvection.

We use several redundancies between different forms of the relations R5 and R6 to reduce the number of computations.
Inverting the parameter $x_b^\beta$ in R5 (as it appears in Table \ref{table:basicrelsian}) is the same as cyclically permuting the relation.
Inverting the parameter $x_a^\alpha$ in R6 is the same as applying a relation from R2 to the original R6 relation.
Swapping the roles of $x_b^\beta$ and $x_c^\gamma$ in R6 is the same as inverting and cyclically permuting the original R6 relation and applying a relation from R2.

We use the identity:
\[\Con{x_a}{x_b}^{-\beta}\ttheta(\Mul{x_a^\alpha}{x_b}^\beta)(t)\Con{x_a}{x_b}^\beta=\ttheta(\Mul{x_a^{-\alpha}}{x_b}^{-\beta})(t)\]
which holds in $\IA_n$ for any $t\in \IA_n$.
This is a consequence of Proposition~\ref{proposition:autaction}.
In particular, this means that we only need to consider one of $\ttheta(\Mul{x_a^\alpha}{x_b^\beta})(r)$ and $\ttheta(\Mul{x_a^{-\alpha}}{x_b^{-\beta}})(r)$; if one is trivial then so is the other.

Since the computations in \verb+rewritethetaR5R6+ rewrite all configurations of $\ttheta(s)(r)$ for $r$ an R5 or R6 relation, up to these reductions, this proves the lemma.
\end{proof}

The next lemma allows us to deal with certain combinations of R5 and R6 relations.
The ordered triple of generators
of $F_n$ {\em involved} in a commutator transvection $\Mulcomm{x_i^{\alpha}}{x_j^{\beta}}{x_k^{\gamma}}$ is 
$(x_i,x_j,x_j)$.
There are eight commutator transvections involving a given triple of generators.

\begin{lemma}\label{lemma:R5R6H2}
Fix distinct $1 \leq a,b,c \leq n$, and let $w \in R\cap[F,F]$ be a product of 
R5 and R6 relations whose commutator transvections involve only $(x_a, x_b, x_c)$, in order.
Then $\Elt{w}$ can be written as a sum of elements of the form $\Elt{v}$ with $v$ an H2 relation.
\end{lemma}
\begin{proof}
\begin{table}[t!]
\begin{tabular}{cl}
Name & Generator \\
\hline
$v_1$ & $\Mulcomm{x_a}{x_b}{x_c}$\\
$v_2$ & $\Mulcomm{x_a^{-1}}{x_b}{x_c}$\\
$v_3$ & $\Mulcomm{x_a}{x_b^{-1}}{x_c}$ \\
$v_4$ & $\Mulcomm{x_a^{-1}}{x_b^{-1}}{x_c}$\\
$v_5$ & $\Mulcomm{x_a}{x_b}{x_c^{-1}}$\\
$v_6$ & $\Mulcomm{x_a^{-1}}{x_b}{x_c^{-1}}$\\
$v_7$ & $\Mulcomm{x_a}{x_b^{-1}}{x_c^{-1}}$\\
$v_8$ & $\Mulcomm{x_a^{-1}}{x_b^{-1}}{x_c^{-1}}$
\end{tabular}
\begin{tabular}{cl}
Name & Relation \\
\hline
$r_1$ & $\Con{x_a}{x_b}\Mulcomm{x_a}{x_b}{x_c}\Con{x_a}{x_b}^{-1}\Mulcomm{x_a}{x_b^{-1}}{x_c}$\\
$r_2$ & $\Con{x_a}{x_b}\Mulcomm{x_a^{-1}}{x_b}{x_c}\Con{x_a}{x_b}^{-1}\Mulcomm{x_a^{-1}}{x_b^{-1}}{x_c}$\\
$r_3$ & $\Con{x_a}{x_c}^{-1}\Mulcomm{x_a}{x_b}{x_c}\Con{x_a}{x_c}\Mulcomm{x_a}{x_b}{x_c^{-1}}$\\
$r_4$ & $\Con{x_a}{x_c}^{-1}\Mulcomm{x_a^{-1}}{x_b}{x_c}\Con{x_a}{x_c}\Mulcomm{x_a^{-1}}{x_b}{x_c^{-1}}$\\
$r_5$ & $\Con{x_a}{x_c}^{-1}\Mulcomm{x_a}{x_b^{-1}}{x_c}\Con{x_a}{x_c}\Mulcomm{x_a}{x_b^{-1}}{x_c^{-1}}$\\
$r_6$ & $\Con{x_a}{x_c}^{-1}\Mulcomm{x_a^{-1}}{x_b^{-1}}{x_c}\Con{x_a}{x_c}\Mulcomm{x_a^{-1}}{x_b^{-1}}{x_c^{-1}}$\\
$r_7$ & $\Con{x_a}{x_b}\Mulcomm{x_a}{x_b}{x_c^{-1}}\Con{x_a}{x_b}^{-1}\Mulcomm{x_a}{x_b^{-1}}{x_c^{-1}}$\\
$r_8$ & $\Con{x_a}{x_b}\Mulcomm{x_a^{-1}}{x_b}{x_c^{-1}}\Con{x_a}{x_b}^{-1}\Mulcomm{x_a^{-1}}{x_b^{-1}}{x_c^{-1}}$\\
$r_9$ & $\Mulcomm{x_a}{x_b}{x_c}\Mulcomm{x_a^{-1}}{x_b}{x_c}[\Con{x_a}{x_b}^{-1},\Con{x_a}{x_c}^{-1}]$ \\
$r_{10}$ & $\Mulcomm{x_a}{x_b^{-1}}{x_c}\Mulcomm{x_a^{-1}}{x_b^{-1}}{x_c}[\Con{x_a}{x_b},\Con{x_a}{x_c}^{-1}]$ \\
$r_{11}$ & $\Mulcomm{x_a}{x_b}{x_c^{-1}}\Mulcomm{x_a^{-1}}{x_b}{x_c^{-1}}[\Con{x_a}{x_b}^{-1},\Con{x_a}{x_c}]$ \\
$r_{12}$ & $\Mulcomm{x_a}{x_b^{-1}}{x_c^{-1}}\Mulcomm{x_a^{-1}}{x_b^{-1}}{x_c^{-1}}[\Con{x_a}{x_b},\Con{x_a}{x_c}]$ \\
$r_{13}$ & $\Mulcomm{x_a^{-1}}{x_b}{x_c}\Mulcomm{x_a}{x_b}{x_c}[\Con{x_a}{x_b}^{-1},\Con{x_a}{x_c}^{-1}]$ \\
$r_{14}$ & $\Mulcomm{x_a^{-1}}{x_b^{-1}}{x_c}\Mulcomm{x_a}{x_b^{-1}}{x_c}[\Con{x_a}{x_b},\Con{x_a}{x_c}^{-1}]$ \\
$r_{15}$ & $\Mulcomm{x_a^{-1}}{x_b}{x_c^{-1}}\Mulcomm{x_a}{x_b}{x_c^{-1}}[\Con{x_a}{x_b}^{-1},\Con{x_a}{x_c}]$ \\
$r_{16}$ & $\Mulcomm{x_a^{-1}}{x_b^{-1}}{x_c^{-1}}\Mulcomm{x_a}{x_b^{-1}}{x_c^{-1}}[\Con{x_a}{x_b},\Con{x_a}{x_c}]$ \\
\end{tabular}
\caption{
Labels for the eight commutator transvections using $x_a$, $x_b$ and $x_c$ in order, and for the sixteen R5 and R6 relations using these commutator transvections.
}\label{table:labelrelations}
\end{table}

Let $F'$ be the subgroup of $F$ generated by the eight commutator transvections involving 
$(x_a,x_b,x_c)$ and the two conjugation moves $\{\Con{x_a}{x_b},\Con{x_a}{x_c}\}$, and let $R' \subset F'$ be the normal closure in $F'$ of the R5 and R6 relations that
can be written as products of elements of $F'$.  
We thus have $w \in R' \cap [F',F']$.  
Our first
goal is to better understand $(R' \cap [F',F'])/[F',R']$ and $R' / [F',R']$.
Consider the exact sequence of abelian groups
\[0 \longrightarrow \frac{R'\cap[F',F']}{[F',R']} \longrightarrow \frac{R'}{[F',R']} \longrightarrow \frac{R'}{R'\cap[F',F']} \longrightarrow 0.\]
We find generators for the first group in the sequence by considering a related exact sequence of free abelian groups.

Let $v_1,\ldots,v_8$ be the eight commutator transvections in $S_{\IA}(n)$ that only involve $(x_a,x_b,x_c)$, enumerated
as in Table~\ref{table:labelrelations}.  
Similarly,
let $r_1,\ldots,r_8$ be the eight R5 relations in $F'$ and let $r_9,\ldots,r_{16}$ denote the eight R6 relations lying in $F'$,
enumerated as in Table~\ref{table:labelrelations}.
Let $A$ be the free abelian group freely generated by $r_1,\ldots, r_{16}$,
and let $B$ be the free abelian group freely generated by $v_1,\ldots,v_8$.
We consider the map $A\to B$ that counts the exponent-sum of each commutator transvection generator.
Let $C$ denote the kernel of this map and let $B'$ denote the image.

Since $R'$ is normally generated by relations $r_1,\ldots,r_{16}$, we know that $R'/[F',R']$ is generated by the 
images of these relations.
Thus there is a surjection $A\to R'/[F',R']$ that sends each basis element to the image of the relation with the same name.
The group $B$ is a subgroup of $F'/[F',F']$.
The natural map $R'/[F',R']\to F'/[F',F']$ counts exponent-sums of generators.
Since the generators $r_1,\ldots,r_{16}$ all have zero exponent-sum for conjugation move generators, 
we do not lose any information by counting only commutator transvection generators in $B$.
This means we have a commuting square
\[
\xymatrix{
A \ar[r]\ar[d] & B \ar[d] \\
\frac{R'}{[F',R']} \ar[r] & \frac{F'}{[F',F']}.
}
\]
The subgroup $B'$ thus maps surjectively onto $(R'[F',F'])/[F',F']$, which is isomorphic to $R'/(R'\cap[F',F'])$.
Therefore we have a commuting diagram with exact rows
\[
\xymatrix{
0 \ar[r] & C \ar[r] \ar[d]                & A\ar[r] \ar[d]         & B' \ar[r] \ar[d]            & 0  \\
0 \ar[r] & \frac{R'\cap[F',F']}{[F',R']}\ar[r] & \frac{R'}{[F',R']} \ar[r] & \frac{R'}{R'\cap[F',F']} \ar[r] & 0 \\
}
\]

The map $B'\to F'/[F',F']$ is injective, so $B'\to R'/(R'\cap[F',F'])$ is also injective.
By construction, $A\to R'/[F',R']$ is surjective.
It follows from a simple diagram chase that $C\to (R'\cap[F',F'])/[F',R']$ is surjective.

The map $A\to B$ is given by this $8\times 16$ matrix:
\[
\left(
\begin{array}{cccccccccccccccc}
1 & 0 & 1 & 0 & 0 & 0 & 0 & 0 & 1 & 0 & 0 & 0 & 1 & 0 & 0 & 0 \\ 
0 & 1 & 0 & 1 & 0 & 0 & 0 & 0 & 1 & 0 & 0 & 0 & 1 & 0 & 0 & 0 \\ 
1 & 0 & 0 & 0 & 1 & 0 & 0 & 0 & 0 & 1 & 0 & 0 & 0 & 1 & 0 & 0 \\ 
0 & 1 & 0 & 0 & 0 & 1 & 0 & 0 & 0 & 1 & 0 & 0 & 0 & 1 & 0 & 0 \\ 
0 & 0 & 1 & 0 & 0 & 0 & 1 & 0 & 0 & 0 & 1 & 0 & 0 & 0 & 1 & 0 \\ 
0 & 0 & 0 & 1 & 0 & 0 & 0 & 1 & 0 & 0 & 1 & 0 & 0 & 0 & 1 & 0 \\ 
0 & 0 & 0 & 0 & 1 & 0 & 1 & 0 & 0 & 0 & 0 & 1 & 0 & 0 & 0 & 1 \\ 
0 & 0 & 0 & 0 & 0 & 1 & 0 & 1 & 0 & 0 & 0 & 1 & 0 & 0 & 0 & 1 \\ 
\end{array}
\right)
\]	
A straightforward linear algebra computation shows that $C$, the kernel of this map, is generated by the following nine 
vectors:
\[r_1-r_3-r_5+r_7,\ r_2-r_4-r_6+r_8,\ -r_1-r_2+r_{13}+r_{14}\]
\[-r_3-r_4+r_{13}+r_{15}, r_1+r_2-r_5-r_6-r_{13}+r_{12},\]
\[r_9-r_{13},\ r_{10}-r_{14},\ r_{11}-r_{15},\text{ and } r_{12}-r_{16}.\]
Since $C$ surjects on $(R'\cap [F',F'])/[F',R']$, we know that $(R'\cap [F',F'])/[F',R']$ is generated by the images of these nine elements.

We will now describe calculations that show that each of the generators above is equivalent modulo $[F',R']$ to an H2 relation.
In each case, we find representatives in $R'\cap [F',F']$ of the image of the given element of $C$.
Since we are working modulo $[F',R']$ we may conjugate any $r_i$ in computing the representative.
In reducing to an H2 relation, we may also apply any relation at all, as long as we apply its inverse somewhere else.

The last four generators are easily equivalent to H2 relations.
We skip the second kernel generator because its image is equal to the first after inverting $x_a$, and we skip the 
fourth because because its image is equal to the third after swapping $x_b$ and $x_c$.
The three computations in the list \verb+kernellist+ finish the lemma by showing that the first, third and fifth generators 
are equivalent to H2 relations.
\end{proof}

\begin{proof}[Proof of Theorem~\ref{theorem:h2gen}]
We must show that every element of $\HIAT$ that happens to lie in $\HH_2(\IA_n)$ can be written as a sum
of elements of 
\[\Set{$\Elt{\ttheta(w)(r)}$}{$w \in S_{\Aut}(n)^{\ast}$ and $r$ is one of the relations H1--H9 from Table \ref{table:commutatorrelators}}.\]
Combining Lemmas \ref{lemma:hiat} and \ref{lemma:rewritethetaR5R6} with the fact that R0 and R5 and R6 are the only
relations in our L-presentation for $\IA_n$ that do not appear as one of the commutator relations
in Table \ref{table:commutatorrelators}, we see that it enough to deal with
sums of elements of the set
\[\Set{$\Elt{r}$}{$r$ is one of the relations R0, R5, and R6}.\]
So consider $\Elt{w} \in \HH_2(\IA_n)$ that can be written
\[\Elt{w} = \sum_{i=1}^m \Elt{r_i}\]
with each $r_i$ either an R0 or R5 or R6 relation.

For any choice of distinct $1\leq a,b,c\leq n$, we consider the commutator transvection generators
involving $(x_a, x_b,x_c)$ or $(x_a,x_c,x_b)$, and the R5, R6 and R0 relations involving only these commutator transvections.
We write 
\[\Elt{w}=\sum_{i=1}^{n+\binom{n-1}{2}}\Elt{w_i},\]
where each $\Elt{w_i}$ is a sum of R5, R6 and R0 relations involving only a single choice of $(x_a,\{x_b,x_c\})$.
To prove the theorem, it is enough to show that we can write each of the $\Elt{w_i}$ as a sum of our generators for $\HH_2(\IA_n)$.
So we assume $\Elt{w}=\Elt{w_i}$ for some $i$; this amounts to fixing a choice of $(x_a,\{x_b,x_c\})$ and assuming $\Elt{w}$ is a sum of R0, R5 and R6 generators using only commutator transvections involving this triple.

We call a commutator transvection $\Mulcomm{x_i^\alpha}{x_j^\beta}{x_k^\gamma}$ \emph{positive}
if $j<k$, and \emph{negative} otherwise.
Each R5 or R6 relation contains two positive commutator transvections (and no negative ones), or two negative ones (and no positive ones).
Suppose $\Elt{r_i}$ is an R5 or R6 relation with negative generators, appearing in the sum defining $\Elt{w}$.
By inserting an R0 relation and its inverse into $r_i$, we replace both of the negative generators with positive ones.
Let $r_i'$ denote the word we get by doing this to $r_i$.
Since we have added and subtracted the same element in $\HIAT$, we have $\Elt{r_i'}=\Elt{r_i}$.
Modifying an R5 or R6 relation in this way gives us the inverse of an R5 or R6 relation involving the same $(x_a,\{x_b,x_c\})$, up to cyclic permutation of the relation.
So we interpret this move as rewriting the sum defining $\Elt{w}$: we replace the relation $\Elt{r_i}$ with the new relation $\Elt{r_i'}$, which is an R5 or R6 relation without negative commutator transvections.
We proceed to eliminate all the negative commutator transvections in R5 and R6 relations in $\Elt{w}$ this way.

Having done this, the only negative commutator transvections the sum defining $\Elt{w}$ appear in R0 relations.
Since $\Elt{w}\in \HH_2(\IA_n)$, the negative generators appear with exponent-sum zero; so the R0 relations appear in inverse pairs.
This means that we can simply rewrite the sum without any R0 relations.
So $\Elt{w}$ is a sum of R5 and R6 relations whose only commutator transvections are positive ones involving $(x_a,\{x_b,x_c\})$.
Then $\Elt{w}$ satisfies the hypotheses of Lemma~\ref{lemma:R5R6H2} and therefore is a sum of H2 generators.
\end{proof}

\section{Coinvariants and congruence subgroups}
\label{section:finalcalcs}

This section contains the proofs of Theorems~\ref{maintheorem:coinvariants} and \ref{maintheorem:congruence}, which
can be found in \S \ref{section:coinvariants} and \ref{section:congruence}, respectively.  Both of these
proofs depend on calculations that are contained in \S \ref{section:glnzaction}.

\subsection{The action of \texorpdfstring{$\GL_n(\Z)$}{GLn(Z)} on \texorpdfstring{$\HH_2(\IA_n)$}{H2(IAn)}}
\label{section:glnzaction}

This section is devoted to understanding the action of $\GL_n(\Z)$ on our generators for
$\HH_2(\IA_n)$.
The results in this section consist of long lists of equations that are verified by a computer,
so on their first pass a reader might want to skip to the next two sections to see how they are used.
For $i=1,\dotsc,9$, we use the notation $h_i(x_{a_1}^{\alpha_1},\dotsc,x_{a_{k_i}}^{\alpha_{k_i}})\in \HH_2(\IA_n)$ for the image in $\HH_2(\IA_n)$ of the $i^{\text{th}}$ relation from from $S_H(n)$, with the given parameters, as specified in Table~\ref{table:commutatorrelators}.
Since the action of $\GL_n(\Z)$ is induced from the action of $\Aut(F_n)$, we record the action of various $\Aut(F_n)$ generators on these generators.  

The computations justifying Lemmas~\ref{lemma:7pt2}--\ref{lemma:7pt6} are in the file \verb+h2ia.g+.
We use the Hopf isomorphism $\HH_2(\IA_n)\cong (R\cap[F,F])/[F,R]$, where $F=F(S_{\IA}(n))$ and $R<F$ is the group of relations of $\IA_n$. 
We justify these equations by performing computations in $R\cap[F,F]\subset F$.
In each computation, we start with a word representing one side of the equation and reduce to the trivial word using words representing the other side.
Since $[F,R]$ is trivial, we may use any relations in inverse pairs, we may apply relations from in any order, and we may cyclically permute relations.

We note some identities, which we leave as an exercise.
\begin{lemma} 
\label{lemma:7pt1}
The following identities hold in $\HH_2(\IA_n)$.  The letters in subscripts are assumed distinct unless otherwise noted.
\begin{compactitem}
\item $h_1(x_a^\alpha,x_b^\beta,x_c^\gamma,x_b^\delta)=-h_1(x_a^\alpha,x_b^\beta,x_c^\gamma,x_b^{-\delta})$,
even if $b=d$.
\item $h_3(x_a^\alpha,x_b^\beta,x_c^\gamma,x_b^\delta,x_e^\epsilon)=-h_3(x_a^\alpha,x_b^\beta,x_c^\gamma,x_b^\delta,x_e^{-\epsilon})$,
even if $b=e$ or $c=e$.
\item $h_3(x_a^\alpha,x_b^\beta,x_c^\gamma,x_b^\delta,x_e^\epsilon)=-h_3(x_a^\alpha,x_c^\gamma,x_b^\beta,x_b^\delta,x_e^\epsilon)$,
even if $b=e$ or $c=e$.
\end{compactitem}
\end{lemma}

We also need the following, which is not obvious.
\begin{lemma}
\label{lemma:7pt2}
The following identities hold in $\HH_2(\IA_n)$.  The letters in subscripts are assumed distinct unless otherwise noted.
\begin{compactitem}
\item 
$\begin{aligned}[t]
h_6(x_a^\alpha,x_e^\epsilon,x_c^\gamma,x_d^\delta,x_e^\epsilon) =\ &h_6(x_a^\alpha,x_c^\gamma,x_e^{-\epsilon},x_d^\delta,x_e^\epsilon)\\
&-h_7(x_a^\alpha,x_e^\epsilon,x_d^\delta,x_e^\epsilon)-h_7(x_a^\alpha,x_e^{-\epsilon},x_d^\delta,x_e^\epsilon).\end{aligned}$
\eqnum\label{eq:h6secondparam}
\item 
$h_6(x_a^\alpha,x_b^\beta,x_c^\gamma,x_d^\delta,x_e^\epsilon) = -h_6(x_a^\alpha,x_c^\gamma,x_b^\beta,x_d^\delta,x_e^\epsilon)$,
even if $b=e$ or $c=e$. \eqnum\label{eq:h6secondthird}
\end{compactitem}
\end{lemma}
\begin{proof}
Computations justifying these equations appear in the list \verb+lemma7pt2+.
\end{proof}

We proceed by expressing the action of many elementary matrices from $\GL_n(\Z)$ on our generators.
\begin{lemma}
The following identities hold in $\HH_2(\IA_n)$.  The letters in subscripts are assumed distinct unless otherwise noted.
\begin{compactitem}

\item $\Mul{x_b^\beta}{x_e}^\epsilon \cdot h_1(x_a^\alpha,x_b^\beta,x_c^\gamma,x_d^\delta) =
h_1(x_a^\alpha,x_b^\beta,x_c^\gamma,x_d^\delta)+h_1(x_a^\alpha,x_e^\epsilon,x_c^\gamma,x_d^\delta)$.
\eqnum\label{eq:generich1}

\item $\Mul{x_b^\beta}{x_e}^\epsilon \cdot h_1(x_a^\alpha,x_e^\epsilon,x_c^\gamma,x_d^\delta) =
h_1(x_a^\alpha,x_e^\epsilon,x_c^\gamma,x_d^\delta)$.
\eqnum\label{eq:generich1fixed}

\item $\Mul{x_b^\beta}{x_d}^\delta\cdot h_1(x_a^\alpha,x_b^\beta,x_c^\gamma,x_d^\delta)
= h_1(x_a^\alpha,x_b^\beta,x_c^\gamma,x_d^\delta) + h_1(x_a^\alpha,x_d^\delta,x_c^\gamma,x_d^\delta)$.
\eqnum\label{eq:specialh1}

\item $\Mul{x_a^\alpha}{x_e}^{-\epsilon} \cdot h_1(x_a^\alpha,x_b^\beta, x_c^\gamma,x_d^\delta) 
=
h_1(x_a^\alpha,x_b^\beta, x_c^\gamma,x_d^\delta)  - h_3(x_a^\alpha,x_e^{\epsilon},x_b^\beta,x_c^\gamma,x_d^\delta),$\eqnum\label{eq:h3}
\\
even if $x_e^\epsilon=x_d^\delta$.

\item $\begin{aligned}[t]
\Mul{x_d^\delta}{x_e}^{-\epsilon}\cdot h_3(x_a^\alpha,x_b^\beta,x_c^\gamma,x_d^\delta, x_f^\zeta)
=\ &h_3(x_a^\alpha,x_b^\beta,x_c^\gamma,x_d^\delta, x_f^\zeta)\\
&-h_2(x_a^\alpha,x_b^\beta,x_c^\gamma,x_d^\delta, x_e^\epsilon,x_f^\zeta),
\end{aligned}$\eqnum\label{eq:generich2}\\
even if $\{b,c\}\cap\{e,f\}\neq\varnothing$.

\item $\begin{aligned}[t]
\Mul{x_a^{-\alpha}}{x_d}^\delta \cdot h_2(x_a^\alpha,x_b^\beta,x_c^\gamma,x_d^{-\delta},x_f^\zeta,x_e^\epsilon)
=\ &h_2(x_a^\alpha,x_b^\beta,x_c^\gamma,x_d^{-\delta},x_f^\zeta,x_e^\epsilon)\\%&-h_2(x_d^{-\delta},x_f^\zeta,x_e^\epsilon,x_a^{-\alpha},x_f^\zeta,x_e^\epsilon)\\
   &-h_2(x_a^\alpha,x_b^\beta,x_c^\gamma,x_a^{-\alpha},x_e^\epsilon, x_f^\zeta),
\end{aligned}$\eqnum\label{eq:bothsidesh2} \\
even if $\{b,c\}\cap\{e,f\}\neq \varnothing$.

\end{compactitem}
\end{lemma}
\begin{proof}
These computations appear in \verb+lemma7pt3+.
The equations where coincidences are allowed are justified in several different computations.
\end{proof}

\begin{lemma}
The following identities hold in $\HH_2(\IA_n)$.  The letters in subscripts are assumed distinct.
\begin{compactitem}
\item $\Mul{x_d^\delta}{x_b}^\beta\cdot h_4(x_a^\alpha,x_d^\delta,x_c^\gamma) =
h_4(x_a^\alpha,x_d^\delta,x_c^\gamma) +h_4(x_a^\alpha,x_b^\beta,x_c^\gamma)$.
\eqnum\label{eq:h4}

\item $\Mul{x_d^\delta}{x_b}^\beta\cdot h_4(x_a^\alpha,x_b^\beta,x_c^\gamma) =
h_4(x_a^\alpha,x_b^\beta,x_c^\gamma)$.
\eqnum\label{eq:fixedh4}

\item $\begin{aligned}[t]
\Mul{x_b^\beta}{x_c}^{-\gamma}\cdot h_1(x_a^\alpha,x_b^\beta,x_c^\gamma,x_d^\delta)
=\ &h_1(x_a^\alpha,x_b^\beta,x_c^\gamma,x_d^\delta) - h_3(x_b^\beta,x_d^\delta,x_c^\gamma,x_a^\alpha,x_c^\gamma)\\
   &+h_5(x_a^\alpha,x_b^\beta,x_c^\gamma,x_d^\delta)+h_4(x_c^\gamma,x_d^\delta,x_a^\alpha).
\end{aligned}$\eqnum\label{eq:h5}

\item $\begin{aligned}[t]
\Mul{x_d^\delta}{x_e}^\epsilon \cdot h_8(x_a^\alpha,x_b^\beta,x_c^\gamma,x_d^\delta)
=\ &h_8(x_a^\alpha,x_b^\beta,x_c^\gamma,x_d^\delta) + h_8(x_a^\alpha,x_b^\beta,x_c^\gamma,x_e^\epsilon)\\
   &+\text{(H1 generators).}\end{aligned}$\eqnum\label{eq:h8}

\item $\Mul{x_d^\delta}{x_e}^\epsilon \cdot h_8(x_a^\alpha,x_b^\beta,x_c^\gamma,x_e^\epsilon) =
h_8(x_a^\alpha,x_b^\beta,x_c^\gamma,x_e^\epsilon)$.
\eqnum\label{eq:h8fixed} 

\end{compactitem}
\end{lemma}
\begin{proof}
These computations appear in \verb+lemma7pt4+.
\end{proof}

\begin{lemma}
The following identities hold in $\HH_2(\IA_n)$.  The letters in subscripts are assumed distinct.
\begin{compactitem}

\item $\begin{aligned}[t]
\Mul{x_f^\zeta}{x_e}^{\epsilon}\cdot h_6(x_a^\alpha,x_b^\beta,x_c^\gamma,x_d^\delta,x_f^\zeta)
=\ &h_6(x_a^\alpha,x_b^\beta,x_c^\gamma,x_d^\delta,x_f^\zeta)\\
   &+h_6(x_a^\alpha,x_b^\beta,x_c^\gamma,x_d^\delta,x_e^\epsilon).\end{aligned}$
\eqnum\label{eq:generich6}

\item $\Mul{x_f^\zeta}{x_e}^{\epsilon}\cdot h_6(x_a^\alpha,x_b^\beta,x_c^\gamma,x_d^\delta,x_e^\epsilon)
= h_6(x_a^\alpha,x_b^\beta,x_c^\gamma,x_d^\delta,x_e^\epsilon).$
\eqnum\label{eq:genericfixedh6}

\item $\begin{aligned}[t]
\Mul{x_c^\gamma}{x_e}^{-\epsilon}\cdot h_7(x_a^\alpha,x_b^\beta,x_e^\epsilon,x_d^\delta)
=\ &h_7(x_a^\alpha,x_b^\beta,x_e^\epsilon,x_d^\delta) +h_7(x_a^\alpha,x_b^\beta,x_c^\gamma,x_d^\delta)\\
   &+h_6(x_e^\epsilon,x_b^\beta,x_a^\alpha,x_c^\gamma,x_d^\delta)\\
   &+\text{(H1, H2, and H3 generators).}\end{aligned}$
\eqnum\label{eq:generich7}

\item $\begin{aligned}[t]
\Mul{x_c^\gamma}{x_e}^{-\epsilon}\cdot h_7(x_a^\alpha,x_b^\beta,x_c^\gamma,x_d^\delta)
=\ &h_7(x_a^\alpha,x_b^\beta,x_c^\gamma,x_d^\delta)\\
   &+h_2(x_c^{-\gamma},x_e^{-\epsilon},x_d^{-\delta},x_c^\gamma,x_b^\beta,x_a^\alpha).\end{aligned}$
\eqnum\label{eq:genericfixedh7}

\item $\begin{aligned}[t]
\Mul{x_d^\delta}{x_b}^\beta\cdot h_7(x_a^\alpha,x_d^\delta,x_c^\gamma,x_d^\delta) 
=\ &h_7(x_a^\alpha,x_d^\delta,x_c^\gamma,x_d^\delta) + h_7(x_a^\alpha,x_b^\beta,x_c^\gamma,x_b^\beta)\\
   &+ h_7(x_a^\alpha,x_b^\beta,x_c^\gamma,x_d^\delta) + h_7(x_a^\alpha,x_d^\delta,x_c^\gamma,x_b^\beta)\\
   &+\text{(H1 generators).}\end{aligned}$
\eqnum\label{eq:part1h7}

\item $\Mul{x_d^\delta}{x_b}^\beta\cdot h_7(x_a^\alpha,x_b^\beta,x_c^\gamma,x_b^\beta) =
h_7(x_a^\alpha,x_b^\beta,x_c^\gamma,x_b^\beta).$
\eqnum\label{eq:part1fixedh7}

\item $\begin{aligned}[t]
\Mul{x_d^\delta}{x_b}^\beta \cdot h_7(x_a^\alpha,x_d^{-\delta},x_c^\gamma,x_d^\delta)
=\ &h_7(x_a^\alpha,x_d^{-\delta},x_c^\gamma,x_d^\delta) + h_7(x_a^\alpha,x_b^{-\beta},x_c^\gamma,x_b^\beta)\\
   &+ h_7(x_a^\alpha,x_d^{-\delta},x_c^\gamma,x_b^\beta) - h_7(x_a^\alpha,x_b^{-\beta},x_c^\gamma,x_d^{-\delta}).
\end{aligned}$
\eqnum\label{eq:part2h7}

\item $\Mul{x_d^\delta}{x_b}^\beta \cdot h_7(x_a^\alpha,x_b^{-\beta},x_c^\gamma,x_b^\beta) =
h_7(x_a^\alpha,x_b^{-\beta},x_c^\gamma,x_b^\beta).$
\eqnum\label{eq:part2fixedh7}

\item $\begin{aligned}[t]
\Mul{x_f^\zeta}{x_e}^\epsilon \cdot h_6(x_a^\alpha, x_f^\zeta, x_c^\gamma, x_d^\delta, x_f^\zeta)
=\ &h_6(x_a^\alpha, x_f^\zeta, x_c^\gamma, x_d^\delta, x_f^\zeta) -h_7(x_a^\alpha, x_e^\epsilon, x_d^\delta, x_f^\zeta)\\
&-h_7(x_a^\alpha, x_e^\epsilon, x_d^\delta, x_e^\epsilon)+h_6(x_a^\alpha, x_f^\zeta, x_c^\gamma, x_d^\delta, x_e^\epsilon)\\
&-h_7(x_a^\alpha, x_e^{-\epsilon}, x_d^\delta, x_f^\zeta)-h_7(x_a^\alpha, x_e^{-\epsilon}, x_d^\delta, x_e^\epsilon)\\
&+h_6(x_a^\alpha, x_e^\epsilon, x_c^\gamma, x_d^\delta, x_f^\zeta)+h_6(x_a^\alpha, x_e^\epsilon, x_c^\gamma, x_d^\delta, x_e^\epsilon).\end{aligned}$
\eqnum\label{eq:specialh6}

\item $\Mul{x_f^\zeta}{x_e}^\epsilon \cdot h_6(x_a^\alpha, x_e^\epsilon, x_c^\gamma, x_d^\delta, x_e^\epsilon) =
h_6(x_a^\alpha, x_e^\epsilon, x_c^\gamma, x_d^\delta, x_e^\epsilon).$
\eqnum\label{eq:specialfixedh6}

\end{compactitem}
\end{lemma}
\begin{proof}
These computations appear in \verb+lemma7pt5+.
\end{proof}

\begin{lemma}\label{lemma:7pt6}
The following identities hold in $\HH_2(\IA_n)$.  The letters in subscripts are assumed distinct.
\begin{compactitem}

\item $\begin{aligned}[t]
\Mul{x_b^\beta}{x_d}^\delta\cdot h_9(x_a^\alpha,x_b^\beta,x_c^\gamma)
=\ &h_9(x_a^\alpha,x_b^\beta,x_c^\gamma) +h_9(x_a^\alpha,x_d^\delta,x_c^\gamma)\\
   &+\text{ (H1--H5 generators).}\end{aligned}$
\eqnum\label{eq:h9}

\item 
$\Mul{x_b^\beta}{x_d}^\delta\cdot h_9(x_a^\alpha,x_d^\delta,x_c^\gamma)
= h_9(x_a^\alpha,x_d^\delta,x_c^\gamma)
+\text{ (H1--H6 generators).}$
\eqnum\label{eq:h9fixed}

\end{compactitem}
\end{lemma}
\begin{proof}
These computations are contained in the list \verb+lemma7pt6+, but they take some explanation.
We use the relations \verb+exiarel(5,[...])+ frequently in these computations; these relations are always combinations of H1 and H4 relations.
We use \verb+exiarel(7,[xa,xb,xc])+ to represent $h_9(x_a^\alpha,x_b^\beta,x_c^\gamma)$, but we also use other relations in this computation.
The relation \verb+exiarel(8,[xa,xb,xc,xd])+ is an expanded version of this relation that behaves better under this action.
Its derivation in \verb+exiarelchecklist+ shows that it differs from $-h_9(x_a^\alpha,x_b^\beta,x_c^\gamma)$ only by H1, H4 and H5 relations.
We also use \verb+exiarel(6,[xa,xb,xc,xd])+; this differs from $h_9(x_a^\alpha,x_b^\beta,x_c^\gamma)$ by H1 and H4 relations, and an R6 relation.
This is apparent in the derivation of \verb+exiarel(7,[xa,xb,xc])+ in \verb+exiarelchecklist+.

The computation justifying equation~\eqref{eq:h9} shows directly that the image of the relation \verb+exiarel(8,[xa,xb,xc,xd])+ under \verb+[["M",xb,xd]]+ can be reduced to the trivial word by applying:
\begin{compactitem}
\item \verb+exiarel(6,[xa,xd,xc])+,
\item \verb+iw(exiarel(7,[xa,xb,xc]))+, 
\item H1--H5 relations (some in \verb+exiarel(5,[...])+ relations), 
\item R5 and R6 relations, and 
\item elements from $[F,R]$ (including inverse pairs of instances of \verb+exiarel(1,[...])+ and \verb+exiarel(3,[...])+).
\end{compactitem}
Since we start and end with elements of $R\cap[F,F]$ in this computation, the use  of R5 and R6 (in one case inside \verb+exiarel(6,[xa,xb,xc,xd])+) is inconsequential; by Lemma~\ref{lemma:R5R6H2} this can only change the outcome by H2 relations.
So this proves equation~\eqref{eq:h9}.

The computation justifying equation~\eqref{eq:h9fixed} is similar, but uses an instance of the relation \verb+exiarel(2,[...])+.
The derivation in \verb+exiarelchecklist+ shows that this relation is a combination of H2, H5, and H6 relations, and elements of $[F,R]$. 
\end{proof}

\subsection{Coinvariants of \texorpdfstring{$\HH_2(\IA_n)$}{H2(IAn)}}
\label{section:coinvariants}

In this section, we prove Theorem \ref{maintheorem:coinvariants}, which asserts that for the $\GL_n(\Z)$-coinvariants
of $\HH_2(\IA_n)$ vanish for $n \geq 6$.
\begin{proof}[Proof of Theorem~\ref{maintheorem:coinvariants}]
We use the generators H1--H9 from Table~\ref{table:commutatorrelators}.
To show that the coinvariants $\HH_2(\IA_n)_{\GL_n(\Z)}$ are trivial, we show that the coinvariance class of each of these generators is trivial.
The covariants are defined to be the largest quotient of $\HH_2(\IA_n)$ with a trivial induced action of $\GL(n,\Z)$.
Since the action of $\GL(n,\Z)$ on $\HH_*(\IA_n)$ is induced by the action of $\Aut(F_n)$ on $\IA_n$, this means that
$\HH_2(\IA_n)_{\GL_n(\Z)}$ is the quotient of $\HH_2(\IA_n)$ by the subgroup generated by classes of the form $f\cdot r-r$, where $f\in \Aut(F_n)$ and $r\in\HH_2(\IA_n)$.
Elements of the form $f\cdot r-r$ are called \emph{coboundaries}.

In fact, we have already shown in Lemmas~\ref{lemma:7pt1}--\ref{lemma:7pt6} that the subgroup of  $\HH_2(\IA_n)$ generated by coboundaries contains our generators from Theorem~\ref{theorem:h2gen}.
We show this for each generator in the equations above as follows.
Each equation shows how to express the given generator as a sum of coboundaries and generators previously expressed in terms of coboundaries:
\begin{compactitem}
\item H1 in equations~\eqref{eq:generich1} and~\eqref{eq:specialh1}, also using an observation from Lemma~\ref{lemma:7pt1},
\item H2 in equations~\eqref{eq:generich2} and~\eqref{eq:bothsidesh2},
\item H3 in equation~\eqref{eq:h3}, also using an observation from Lemma~\ref{lemma:7pt1},
\item H4 in equation~\eqref{eq:h4},
\item H5 in equation~\eqref{eq:h5},
\item generic H6 in equation~\eqref{eq:generich6},
\item H7 in equations~\eqref{eq:generich7}, \eqref{eq:part1h7}, and~\eqref{eq:part2h7},
\item the special cases of H6 in equation~\eqref{eq:specialh6}, also using equations~\eqref{eq:h6secondparam} and~\eqref{eq:h6secondthird},
\item H8 in equation~\eqref{eq:h8} and
\item H9 in equation~\eqref{eq:h9}. \qedhere
\end{compactitem}
\end{proof}
\begin{remark}
The equations above assume that distinct subscripts label distinct elements.
This means that equation~\eqref{eq:generich6} requires six basis elements.
We do not know if the generic H6 generator (a five-parameter generator) can be expressed as a sum of coboundaries without using a sixth basis element.
Therefore we require $n\geq 6$ in the statement and we do not know if the theorem holds for smaller $n$.
\end{remark}

\subsection{Second homology of congruence subgroups}
\label{section:congruence}

In this section, we prove Theorem \ref{maintheorem:congruence}, which asserts that $\HH_2(\Aut(F_n,\ell);\Q)=0$ for $n \geq 6$ and $\ell \geq 2$.
The key to this is the following lemma.  Let $\GL_n(\Z,\ell)$ be the level-$\ell$ congruence subgroup of $\GL_n(\Z)$, that is, the kernel
of the natural map $\GL_n(\Z) \rightarrow \GL_n(\Z/\ell)$.

Like in Theorem~\ref{maintheorem:coinvariants}, we require $n\geq 6$ because of equation~\eqref{eq:generich6}.  We do not know if the result holds for smaller $n$.

\begin{lemma}
\label{lemma:rationalh2coinv}
For $n \geq 6$ and $\ell \geq 2$ we have $(\HH_2(\IA_n;\Q))_{\GL_n(\Z,\ell)} = 0$.
\end{lemma}
\begin{proof}
Again we use our generators from Theorem~\ref{theorem:h2gen}.
The universal coefficient theorem implies that $\HH_2(\IA_n;\Q)$ is generated by the images of our generators from $\HH_2(\IA_n)$.

We have two approaches for showing that a generator has trivial image.
The first is the following: if $f\in \Aut(F_n)$ and $r,s\in \HH_2(\IA_n)$ with $f\cdot r-r=s$ and $f\cdot s=s$ in $(\HH_2(\IA_n;\Q))_{\GL_n(\Z,\ell)}$, then 
\[f^\ell\cdot r-r=\ell s \quad\text{in $(\HH_2(\IA_n;\Q))_{\GL_n(\Z,\ell)}$}.\]
If further, $f^\ell$ lies in $\Aut(F_n,\ell)$, then this shows that $s$ is trivial in $(\HH_2(\IA_n;\Q))_{\GL_n(\Z,\ell)}$.

The second approach is simpler:
if $f\in \Aut(F_n)$ and $r,s\in \HH_2(\IA_n)$ with $f\cdot r-r=s$ 
and $r=0$ in $(\HH_2(\IA_n;\Q))_{\GL_n(\Z,\ell)}$, then of course, $s=0$ in $(\HH_2(\IA_n;\Q))_{\GL_n(\Z,\ell)}$.

We show the generators have trivial images as follows:
\begin{compactitem}
\item generic H1 using equations~\eqref{eq:generich1} and~\eqref{eq:generich1fixed} by the first approach,
\item special cases of H1 using equation~\eqref{eq:specialh1} by the second approach, and using Lemma~\ref{lemma:7pt1},
\item H3 by equation~\eqref{eq:h3} using the second approach, and using Lemma~\ref{lemma:7pt1},
\item H2 by equations~\eqref{eq:generich2} and~\eqref{eq:bothsidesh2}, using the second approach,
\item H4 by equations~\eqref{eq:h4} and~\eqref{eq:fixedh4}, using the first approach,
\item H5 by equation~\eqref{eq:h5}, by the second approach,
\item generic H6 by equations~\eqref{eq:generich6} and~\eqref{eq:genericfixedh6}, using the first approach,
\item generic H7 by equations~\eqref{eq:generich7} and~\eqref{eq:genericfixedh7}, by the first approach,
\item special cases of H7 by equations~\eqref{eq:part1h7} and~\eqref{eq:part1fixedh7}, and by equations~\eqref{eq:part2h7} and~\eqref{eq:part2fixedh7}, both by  the first approach,
\item one special case of H6 by equations~\eqref{eq:specialh6} and~\eqref{eq:specialfixedh6}, by the first approach,
\item other special cases of H6 using the first case and equations~\eqref{eq:h6secondparam} and~\eqref{eq:h6secondthird},
\item H8 by equations~\eqref{eq:h8} and~\eqref{eq:h8fixed} by the first approach and
\item H9 by equations~\eqref{eq:h9} and~\eqref{eq:h9fixed} by the first approach. \qedhere
\end{compactitem}
\end{proof}

\begin{proof}[{Proof of Theorem \ref{maintheorem:congruence}}]
We examine the Hochschild-Serre spectral sequence associated to the short exact sequence
\begin{equation}
\label{eqn:autfnlex}
1 \longrightarrow \IA_n \longrightarrow \Aut(F_n,\ell) \longrightarrow \GL_n(\Z,\ell) \longrightarrow 1.
\end{equation}
First, the Borel stability theorem \cite{BorelStability1} implies that $\HH_2(\GL_n(\Z,\ell);\Q)=0$.
Next, recall from the introduction that $\HH_1(\IA_n;\Q) \cong \Hom(\Q^n,\bigwedge^2 \Q^n)$; the group $\GL_n(\Z)$ acts
on this in the obvious way.  
Since $\Hom(\R^n,\bigwedge^2 \R^n)$ is an irreducible representation of the algebraic group $\SL_n(\R)$, it follows
from work of Borel \cite[Proposition 3.2]{BorelDensity} that $\HH_1(\IA_n;\Q)$ is an irreducible representation
of $\GL_n(\Z,\ell)$ (we remark that the above reference is one of the steps in the original proof of the Borel
density theorem; the result can also be derived directly from the Borel density theorem).
It thus follows from the extension of the Borel stability theorem to nontrivial coefficient systems \cite{BorelStability2}
that
\[\HH_1(\GL_n(\Z,\ell);\HH_1(\IA_n;\Q)) = 0.\]
Lemma \ref{lemma:rationalh2coinv} says that
\[\HH_0(\GL_n(\Z,\ell);\HH_2(\IA_n;\Q)) \cong (\HH_2(\IA_n;\Q))_{\GL_n(\Z,\ell)} = 0.\]
The $p+q=2$ terms of the Hochschild-Serre spectral sequence associated to \eqref{eqn:autfnlex} thus all vanish,
so $\HH_2(\Aut(F_n,\ell);\Q)=0$, as desired.
\end{proof}

\noindent
\begin{tabular*}{\linewidth}[t]{@{}p{\widthof{Department of Mathematical Sciences, 301 SCEN}+0.50in}@{}p{\linewidth - \widthof{Department of Mathematical Sciences, 301 SCEN} - 0.50in}@{}}
{\raggedright
Matthew Day\\
Department of Mathematical Sciences, 301 SCEN\\
University of Arkansas\\
Fayetteville, AR 72701\\
E-mail: {\tt matthewd@uark.edu}}
&
{\raggedright
Andrew Putman\\
Department of Mathematics\\
Rice University, MS 136 \\
6100 Main St.\\
Houston, TX 77005\\
E-mail: {\tt andyp@math.rice.edu}}
\end{tabular*}

\end{document}